\documentclass[a4paper,11 pt]{amsart}

 \usepackage[sc]{mathpazo}
\usepackage{eulervm}
\usepackage{hyperref}

\usepackage[utf8]{inputenc}
\usepackage[T1]{fontenc}

\usepackage[english]{babel}
\usepackage{mathtools}
\usepackage{braket}
\usepackage{amsmath, amssymb, stmaryrd, mathabx}
\usepackage{enumerate}
\usepackage{tikz-cd}
\usepackage{tikz,float, hyperref, circuitikz}
\usepackage{titlesec, cite}

\titleformat{\section}
{\normalfont\normalsize\bfseries}{\thesection.}{1em}{}
\titleformat{\subsection}
{\normalfont\normalsize\itshape}{\thesubsection.}{1em}{}
\titleformat{\subsubsection}
{\normalfont\normalsize\itshape}{\thesubsubsection.}{1em}{}

\usepackage{tikz}
\usetikzlibrary{arrows.meta}

\newcommand{\cone}{\mathrm{Cone}}

\newcommand{\A}{\mathfrak{A}}
\newcommand{\B}{\mathfrak{B}}
\newcommand{\tilS}{\widetilde{\Sigma}}

\newcommand{\C}{\mathbb{C}}

\newcommand{\ak}{\mathfrak{a}_k}

 \theoremstyle{plain}
 \newtheorem{theorem}{Theorem}[section]
\newtheorem{proposition}[theorem]{Proposition}
 \newtheorem{lemma}[theorem]{Lemma}
 
\newtheorem{corollary}[theorem]{Corollary}
\theoremstyle{definition}
 
\newtheorem{example}{Example}

 \newtheorem{remark}[theorem]{Remark}

\title{Intersection cohomology of type-A toric varieties}
\author{Andras Szenes and Olga Trapeznikova}
\address{Section de mathématiques, Université de Genève}
\email{Andras.Szenes@unige.ch}
\address{Section de mathématiques, Université de Genève}
\email{Olga.Trapeznikova@unige.ch}
\begin{document}	

\begin{abstract} 
Type-A toric varieties may be obtained as GIT quotients with respect to a torus action with weights corresponding to roots of the group $ SL(k) $ for some $ k>1 $. These varieties appear in various important applications, in particular, as normal cones to strata in  moduli spaces of vector bundles.
In this paper, we describe the intersection Betti numbers of these varieties, and those of some associated projective varieties. We present an elegant combinatorial model for these numbers, and, using the work of Hausel and Sturmfels, we show that the relevant intersection cohomology groups are endowed with a canonical product structure.
\end{abstract}

\maketitle

\section{Introduction}\label{sec:intro} 

One of the best ways to construct toric varieties is via Geometric 
Invariant Theory  
toric quotients of vector spaces.  Let $ 
\alpha_1,\dots,\alpha_n\in\mathrm{Hom}(T,\C^*) $ be $ n $ weights of a complex 
torus of dimension $ k-1 $, and let $ x_1,..., x_n $ be the  coordinates 
on $ \C^n $. Considering the weights as elements of an additive group, and 
using the exponential notation, we can write the corresponding diagonal action 
of an element $ q\in T $ as
\[    q\cdot (x_1,\dots, x_n)  \mapsto ( q^{\alpha_1}x_1,\dots,q^{\alpha_n} 
x_n).
 \]
The ring of $ T $-invariant polynomial functions on $ \C^n $ then is the ring 
of functions of an affine algebraic variety, which we denote by $X=X(\A,0)$, 
where  $ \A=[\alpha_1,\dots,\alpha_n ]$. 
In our constructions, we will always assume that the sequence
$ \A$ is unimodular, i.e. any $ (k-1) $-tuple of these 
weights is either linearly dependent, or has determinant $ \pm1 $ with respect 
to the lattice $\mathbb{Z}\A$ they generate. For the purposes of this introduction, 
we will consider the case when $ n= k(k-1) $, and $ \A $ is the root system of 
the group $ SL(k) $:
\[ \A = [ \alpha_{ij}=\varepsilon_i - \varepsilon_j|\; 1\le i\neq j\le k ]. 
\]
where $ \varepsilon_i $, $ i=1,\dots k $, are the "coordinate weights" of 
$ (\C^*)^k $ acting on $ \C^k$. 

A variant of this construction is obtained by fixing a weight $ \theta\in\mathbb{N}\A
$ and considering the graded algebra 
\[    \oplus_{j=1}^\infty S_{j\theta}, \]
where $  S_{j\theta}$ is the linear span of the monomials in $\C 
[x_1,..., x_n] $ of weight $ j\theta $. The corresponding variety $ 
X(\A,\theta) $ is smooth for generic $ \theta $, and projective over ${X}(\A,0)$, 
which means, in particular, that there is a canonical proper map 
$$ \varphi_\theta:X(\A,\theta) \to X(\A,0).$$ We note that for every $ 
\eta\in\mathbb{Z}\A  $, 
there is a well defined line bundle $ L_\eta $ on $ X(\A,\theta) $ for each $ \theta$, and  the line bundle 
$ L_\theta\to X(\A,\theta) $ is the polarization for the projective map $ \phi_\theta $.

We note that this variety appears in various contexts, in particular as a quiver variety \cite{Nakajima}, 
and plays an 
important role in the our calculation of the intersection cohomology of the 
moduli spaces of semistable bundles  \cite{MeinRein, MozgovoyRein}.

The quotient variety $ X $ is usually very singular, and the topological 
invariants most adapted to this situation are the intersection cohomology 
groups $ IH^*(X) $.  The central problem we address in this paper, is 
the calculation of the associated Poincar\'e polynomial 
\[   g(k,t)=  \sum_{i=0}^{[d/2]}g_it^i =  \sum_{i=0}^{[d/2]} \dim IH^{2i}(X)\cdot t^i, \]
where $d=n-k=k(k-2)$.

A related problem is the study of the topology of the projective toric variety 
$ \widehat{X}$,
associated to $ \A $ as follows. Let $ \kappa:\C^*\to\C^* $ be the tautological 
weight of the group $ \C^* $, set $ \hat T = T\times\C^* $, and consider 
the weight data $ \widehat{\A}= [\alpha+\kappa|\, \alpha\in \A] $ for the group $ \hat T $.
Then we set $ \widehat{X}=X(\widehat{\A},\kappa)$, and note that
the variety $ X $ is the cone over $ \widehat{X}$ 
associated to the line bundle $ L_\kappa $, i.e. it may be obtained as the 
total space of $ L_\kappa $ with the zero-section collapsed to a point. 

The properties of intersection cohomology imply that $ IH^*(\widehat{X}) $ 
satisfies Poincaré duality,
and thus the corresponding Poincar\'e polynomial
\[   h(k,t)=   \sum_{i=0}^d h_it^i = \sum_{i=0}^{d} \dim IH^{2i}(\widehat{X})\cdot t^i \]
is palindromic.

Our notation for these polynomials is motivated by the $ g$-and-$ h 
$-polynomial calculus of Stanley 
\cite{Stanley}, of which our polynomials are examples. This, in particular, 
implies that 
\begin{equation}\label{ghrelation}
	g_i = h_i - h_{i-1}\text{ for }  0\le i \le[d/2] .
\end{equation}

Let us describe contents of our paper. In \S\ref{sec:toric} we set the basic notation and describe the necessary concepts of the theory of toric varieties, while in \S\ref{sec:intcoh} give a very brief introduction to intersection cohomology.  Our key result, proved in is \S\ref{sec:zhepol} (cf. also \S\ref{S6.1}), is that for a generic $ 
\theta $, the canonical map $ \varphi_\theta:X(\A,\theta) \to X(\A,0)$ is \textit{small} 
(cf. Theorems \ref{volosataya} and \ref{allsmall}), and as a consequence,
the intersection cohomology group $ IH^*(X) $ is isomorphic to the 
cohomology of the fiber $ H^*(\varphi_\theta^{-1}(0)) $ as an additive group. Using this, 
we give a graph-theoretic interpretation of the coefficients of the polynomial
\[ g(k,t+1) = \hat{g}_0+ \hat{g}_1t+\hat{g}_2t^2+...+\hat{g}_{[d/2]} t^{[d/2]} \]
as follows.  
Let $ \mathcal{G}_{k} $ be the set of oriented graphs with vertex set 
$\{1, 2, ..., k\}$, and for $ G\in\mathcal{G}_{k} $,  denote by \begin{itemize}
\item  $e(G) $ the number 
of edges of $ G $, and by 
\item $ \mathrm{out}_G(m) $ the number of outgoing edges from the vertex $ m $ in $ G $.
\end{itemize}
 We show (cf. Theorem \ref{volosataya}) that 
\begin{multline*}
	 \hat{g}_i = |\{G\in\mathcal{G}_{k}|\;G \text{ is acyclic, connected, with 
	 } \\
e(G)=i+k-1,\,\mathrm{out}_G(1)=0,\, \mathrm{out}_G(m)>0\text { for }m>1\}|.
\end{multline*}
For calculating the polynomial $ g(k, t) $, we present the following recursion (cf. Proposition \ref{recgpoly}).
Let $p(m,t)=1+t+t^2+...+t^{m-1},$
and set $g(1,t)=1$. 
Then 
\begin{equation*}
g(k,t)= \sum_{\substack{J\subset\{2,...,k\} \\ J\neq\emptyset}}  (-1)^{|J|-1}p(k-|J|,t)^{|J|} \cdot g(k-|J|,t).
\end{equation*}

Note that   the function $ g $ determines the 
function $ h $ (cf. \eqref{ghrelation}). Nevertheless, going between the two 
functions is nontrivial, 
because of the truncation involved in the passage (cf. \cite{Stanley}).  
In  \S\ref{sec:hpol} we show that for $ k=3 $, the $ h $-polynomial coincides 
with that of the product of projective spaces (cf. Lemma \ref{hpoly3}), and then present 
an efficient recursion for the general case, which in a certain sense, produces 
a resolution of $ IH^*(\widehat{X}) $ in terms of such products.
We also give a graph-theoretic interpretation of the $ h $-polynomial (cf. Theorem \ref{hgraphs}): we show that 
\[ h(k,t+1) = \hat{h}_0+ \hat{h}_1t+\hat{h}_2t^2+...+\hat{h}_dt^d \] 
with 
  \begin{multline*}
	\hat{h}_i = |\{G\in\mathcal{G}_{k}|\;G \text{ has an oriented cycle,} 
	\, e(G)=i+k,\,
		\\ 
		\text{and there is a path to the vertex 1 from any other vertex} \}|,
\end{multline*}
and setting $h(1,t)=1$, we present a recursion (cf. Theorem \ref{main:h-poly}) for the $h$-polynomial:
\begin{equation*}
h(k,t)= p(k-1,t)^{k}-\sum_{\substack{ (\lambda_1,...,\lambda_s)\vdash \underline{k} \\ s\geq 2, |\lambda_i|\geq 2}} t^{\sum_{i<j}|\lambda_i|\cdot|\lambda_j|} \prod_{i=1}^sh(|\lambda_i|,t),
 \end{equation*}
where we denote by $\underline{k}$ the set $\{1,2,...,k\}$.

Finally, we note that our variety is an example of the Lawrence toric varieties
studied by Hausel and Sturmfels \cite{HSt}. They showed, in particular, 
that the 
cohomology rings of the varieties $\varphi^{-1}_\theta(0)$ 
are identical for different 
regular values of $ \theta $, even though these varieties are not, in general, 
all isomorphic. This allows us to define a canonical ring structure on the intersection cohomology 
$ IH^*(X) $  (cf. Theorem \ref{algrelations}).

\textbf{Acknowledgements.} We are grateful to Camilla Felisetti, Nicolas Hemelsoet,  G\'abor 
Hetyei and Tam\'as Hausel for useful discussions.
This research was supported by SNF grant S110100, and the NCCR SwissMAP.

\section{Preliminaries: toric varieties}\label{sec:toric}

This section contains the notation and basic facts from the theory of toric varieties. For more details we refer to \cite{Fulton, SzV}.

\subsection{The quotient construction}\label{S2.1}

 Let $\mathfrak{g}=\oplus_{i=1}^n\mathbb{R}\omega_i$ be a real vector space with a fixed ordered basis, and let 
\begin{equation}\label{seq:toric}
0\to\mathfrak{a}\to \mathfrak{g}\xrightarrow{B} \mathfrak{t}\to 0
\end{equation}
be an exact sequence of finite dimensional real vector spaces of dimensions $n-d, n$ and $d$. We denote by $\Gamma_{\mathfrak{g}} = \oplus_{i=1}^n\mathbb{Z}\omega_i$ the lattice in $\mathfrak{g}$, and assume that $\Gamma_{\mathfrak{g}}$ intersects $\mathfrak{a}$ in a lattice $\Gamma_{\mathfrak{a}}$ of full rank; we denote the image $B(\Gamma_{\mathfrak{g}})$ in $\mathfrak{t}$ by $\Gamma_{\mathfrak{t}}$. The sequence \eqref{seq:toric} restricted to the lattices is also exact, as well as the dual sequence
\begin{equation*}\label{seq:toricdual}
0\to\mathfrak{t}^*\to \mathfrak{g}^*\xrightarrow{A} \mathfrak{a}^*\to 0
\end{equation*}
restricted to the dual lattices $\Gamma_{\mathfrak{t}}^*, \Gamma_{\mathfrak{g}}^* $ and $\Gamma_{\mathfrak{a}}^*$. 
Denoting the dual basis by $\{\omega^i\}$, we have $\mathfrak{g}^*=\oplus_{i=1}^n\mathbb{R}\omega^i$ and $\Gamma_{\mathfrak{g}}^* = \oplus_{i=1}^n\mathbb{Z}\omega^i$. 

We introduce the notation $\alpha_i$ for the vector $A(\omega^i)\in 
\Gamma_{\mathfrak{a}}^*$ and consider the  sequence $\mathfrak{A} = [\alpha_1, 
..., \alpha_n$]. Note that according to our assumptions, the elements of 
$\mathfrak{A}$ generate $\Gamma_{\mathfrak{a}}^*$ over $\mathbb{Z}$.

The complexified torus $T_{\mathfrak{a}}=\mathrm{Hom}(\Gamma_\mathfrak{a}^*,\mathbb{C}^*)$ acts on $\mathbb{C}^n$ diagonally with weights $(\alpha_1, ..., \alpha_n)$. We will be interested in the quotients of this action in the sense of  geometric invariant theory. Let 
$$S = \mathbb{C}[x_1, ..., x_n], \,\,\,\,\,  \mathrm{deg}(x_i) = 
\alpha_i\in\Gamma_{\mathfrak{a}}^*$$ be the ring of polynomials graded by the 
semigroup $\mathbb{N}\A\subset\Gamma_{\mathfrak{a}}^*$. For $\theta \in 
\mathbb{N}\A$, we denote by $S_\theta$ the vector space of homogeneous 
polynomials in $ S $ of degree $\theta$. Then $S_0$ is a finitely generated 
subalgebra of homogeneous degree zero polynomials  and  $S_\theta$ is a module 
over $S_0$ for any $\theta\in\mathbb{N}\A$. 

We define the affine toric variety $X(\A,0)$ as the affine GIT quotient of $\mathbb{C}^n$ by the torus $T_{\mathfrak{a}}$ action:
$$ X(\A,0) := \mathbb{C}^n\!\sslash^0 T_{\mathfrak{a}}  = 
\mathrm{Spec}(S_0).$$   \label{X_0def}
For any $\theta \in \mathbb{N}\A$, we define the toric variety $X(\A,\theta)$ as the relative projective GIT quotient of $\mathbb{C}^n$ by the  $T_{\mathfrak{a}}$-action:
$$X(\A,\theta) := \mathbb{C}^n\sslash^\theta{T_{\mathfrak{a}}} : = \mathrm{Proj}(S_{(\theta)}),$$
where $S_{(\theta)}$ is the finitely generated $S_0$-algebra $\oplus_r t^r S_{r\theta}$, which is $\mathbb{N}$-graded by the degree of $t$.

\subsection{Gale duality and toric fans}\label{S2.2} 
Recall that any toric variety $X$ may be associated to a fan $\Sigma$ in the lattice $\Gamma_\mathfrak{t}\subset\mathfrak{t}$ in such a way that each cone $\sigma\in\Sigma$ corresponds to an affine subset of  $X$ (cf. \cite[\S1.4]{Fulton}). 
In this section we describe the fan which corresponds to the toric variety $X(\A,\theta)$ defined above (for the result, see Proposition \ref{fanforGIT}).

\subsubsection{Chambers in the $ \A $-picture}
We begin with the definition of some basic concepts related to our weight 
sequence $\A=[\alpha_1, ...,\alpha_n]$.

\begin{itemize}\label{bind}

\item For a set or sequence $S$ of vectors in a real vector space, denote by 
$\mathrm{Cone}(S)$ the closed cone spanned by the elements of $S$. 
 By convention, the cone over the empty set is the origin of the vector space.

\item Denote by $\mathrm{BInd}(\mathfrak{A})$ the set of \textit{basis index 
sets}, i.e. the set of those subsets $I \subset \{1, ..., n\}$ for which the 
set $\{\alpha_i\}_{i\in I}$ is a basis of $\mathfrak{a}^*$. We will use the 
notation $\alpha^I\subset \mathfrak{A}$ for the basis associated to $I\in 
\mathrm{BInd}(\mathfrak{A})$.

\item Denote by $ \partial\mathfrak{A} $ the union of the boundaries of the 
simplicial cones spanned by elements of $ \A $:
\[ \partial\mathfrak{A} = \cup\{\partial\cone(\alpha^I)|\,I\in 
\mathrm{BInd}(\mathfrak{A})\}. \]
  A  connected component of the open set $\cone(\A) \setminus 
  \partial\mathfrak{A}$ 
  in $\mathfrak{a}^*$ is called a \textit{chamber}, and the set of chambers will be denoted by $Ch(\A)$.

\item We will call $\theta\in\Gamma_\mathfrak{a}^*$  \textit{generic}, if it 
lies in one of the  chambers of $\A$.

\item For a chamber $\mathfrak{c}\in Ch(\A)$, we define $\mathrm{BInd}(\A, \mathfrak{c})$ to be the set of those $I\in \mathrm{BInd}(\mathfrak{A})$  for which $\cone(\alpha^I)\supset \mathfrak{c}$.

\end{itemize}

The sequence $\mathfrak{B}=[\beta_1,\dots,\beta_n]$, where be  
$\beta_i=B(\omega_i)$ in $\mathfrak{t}$ (cf. \eqref{seq:toric}) is called the 
\textit{Gale dual} of the sequence $\A$. This notion is involutive, i.e. the 
Gale dual of the sequence $\mathfrak{B}$ is $ \A $.

\subsubsection{Gale duality and fans}
A \textit{fan} $ \Sigma $ on $ \B $ is a finite collection of cones  of the form $\sigma= \cone(\beta^I) $, $ I\subset\{1,2,\dots,n\} $, satisfying some additional properties (cf. \cite[\S1.4]{Fulton}); in particular, the union of  ${\sigma\in\Sigma}$ is the cone $ \cone(\B)$.  A fan is \textit{simplicial} if all of its cones are simplicial; further, it is \textit{unimodular} if every maximal cone of $\Sigma$ is spanned by a basis of $\Gamma_\mathfrak{t}$.

There is a standard construction of toric varieties $\Sigma\mapsto X(\Sigma)$ 
from fans in the lattice $\Gamma_\mathfrak{t}$ (cf. \cite[\S1.4]{Fulton}). The 
toric variety $X(\Sigma)$ is a \textit{toric orbifold}, i.e. has only finite 
quotient singularities, if and only if $\Sigma$ is simplicial. In the 
unimodular case, simplicial fans give rise to smooth toric varieties.

The torus $T_{\mathfrak{t}} = \mathrm{Hom}(\Gamma_\mathfrak{t}^*, 
\mathbb{C}^*)$ is embedded in $X(\Sigma)$ as an open subset, and its 
standard action on itself extends to an action on $X(\Sigma)$.

\label{notation:toricorb}
\noindent\textbf{Notation:} The orbits of the $T_{\mathfrak{t}}$-action on 
$X(\Sigma)$ are in bijection with the cones $\sigma\in\Sigma$. Given 
$\sigma\in\Sigma$, we denote by $\mathcal{O}(\sigma)$ the corresponding orbit, 
and by $V(\sigma)$ the orbit closure of $\mathcal{O}(\sigma)$. Note that given $\sigma, \sigma'\in\Sigma$, 
$\sigma\subset\sigma'$ if and only if $V({\sigma'})\subset V(\sigma)$ and 
$\mathrm{dim}(\sigma)=\mathrm{codim}(V(\sigma)\subset X(\Sigma))$. 
\label{orbit-cone}

Next, we say that the fan $\Sigma_1$ is a \textit{refinement} of $\Sigma_2$ if for every cone $\sigma_1\in\Sigma_1$ there exists a cone $\sigma_2\in\Sigma_2$ such that $\sigma_1 \subset \sigma_2$. In this case, we can define a map $\widehat{\psi}_\Sigma: \Sigma_1\to \Sigma_2$ by setting $\widehat{\psi}_\Sigma(\sigma_1)$ to be the smallest cone in $\Sigma_2$ that contains $\sigma_1$, and this induces a so-called  \textit{toric} morphism $\psi: X(\Sigma_1) \to X(\Sigma_2)$.

The following lemma describes the fundamental relation between the Gale dual configurations $\A$ and $\B$. 
\begin{lemma}\label{Galebas}
A linear combination $\sum_i m_i\alpha_i$ vanishes if and only if there is a linear functional $l\in\mathfrak{t}^*$ such that $l(\beta_i)=m_i$.
\end{lemma}
This relationship is instrumental in the proof of the following important statement.
\begin{proposition}\label{fanforGIT}
Let $\A$ be a sequence in $\mathfrak{a}^*$ and let $\mathfrak{c}$ be a chamber in $Ch(\A)$. 
\begin{enumerate}
\item 
 If $I\in \mathrm{BInd}(\A, \mathfrak{c})$, then its complement $\overline{I}=\{1,...,n\}\setminus I$ is an element of $\mathrm{BInd}(\B)$, i.e. the set $\beta^{\overline{I}}=\{\beta_i\}_{i\in\overline{I}}$ is a basis of $\mathfrak{t}$.
 \item The set of simplicial cones $\cone(\beta^{\overline{I}})$, $I\in 
 \mathrm{BInd}(\A, \mathfrak{c})$ forms a simplicial fan  
 $\Sigma(\mathfrak{c})$  on $\B$. In particular, the interiors of these cones 
 are disjoint and their union is $ \cone(\B) $. 
\item  Let $\theta\in \mathfrak{c}$, then the GIT quotient $X(\A,\theta)$ is the toric variety associated to the fan $\Sigma(\mathfrak{c})$.
\item The affine toric variety $X(\A,0)$ is associated to the fan with a single top-dimensional cone: 
 the convex polyhedral cone $\cone({\B})$ spanned by $\B$.
\end{enumerate}
\end{proposition}

\begin{remark}\label{canonocalmor} Recall that for every scheme $X$ there is a canonical morphism
$$\psi: X\to X_0$$ 
to the affine scheme $X_0 = \mathrm{Spec}(H^0(X,\mathcal{O}_X))$ of regular 
functions on $X$. This map for $X(\A,\theta)$ is the canonical map to 
$X(\A,0)$, and, in the unimodular case, this is a resolution of singularities.
\end{remark}

\subsubsection{Polar polytopes}\label{S2.3}

Finally, we note that there is a simple way to associate a fan, and thus a  
toric variety, to a convex polytope. 

Let $P\subset\mathfrak{t}$ be a  \textit{rational} convex polytope, i.e. 
such that an integral multiple of its vertices $ \{v_1,v_2,\dots,v_N\} $ lie in 
the lattice $ \Gamma_\mathfrak{t} $. For simplicity, first, we will assume that $ 
\sum_{i=1}^{N}v_i=0 $. 
Then we denote by $\Sigma_P$ the fan whose cones are the cones over the proper
faces of $P$ (here, we include the empty face). Following the notation from above, we denote by $X(\Sigma_P)$ the associated toric variety. If the center of mass of $ P $ is not at the origin, then we replace the polytope $P$ by its shifted copy $P- 
\frac1N\sum_{i=1}^{N}v_i$. \label{projectivefan}

 More generally, if $\widetilde{P}$ is a subdivision of the boundary of $P$, 
 i.e. $\widetilde{P}$ is a collection of convex polytopes whose union is the 
 boundary of $P$, and the intersection of two polytopes in $\widetilde{P}$ is a 
 polytope in $\widetilde{P}$ (cf. Figure \ref{fig:ch(A)}), then the cones over 
 the polytopes in $\widetilde{P}$ form a fan, which we will also denote by 
 $\Sigma_{\widetilde{P}}$.
Then $\Sigma_{\widetilde{P}}$ is a refinement of $\Sigma_{{P}}$, which induces 
a toric morphism $X(\Sigma_{\widetilde{P}})\to X(\Sigma_{{P}})$.
Note that while $X(\Sigma_P)$ is projective, the variety $ 
X(\Sigma_{\widetilde{P}}) $ is not necessarily projective.

\section{Intersection cohomology}\label{sec:intcoh}
In this section, we collect a few basic facts about small morphisms and the 
intersection cohomology of toric varieties, needed in our paper.  
\subsection{Small maps}\label{S3.1}

Let $ \widetilde{Y}$ be a connected nonsingular  
variety, and $\psi: \widetilde{Y}\to Y$ be a proper surjective map onto a 
variety $ Y $ of the same dimension. 
A \textit{stratification} for $\psi$  is a decomposition of $Y$ into finitely many 
locally closed nonsingular subsets $Y = \bigsqcup_{k=0}^n S_k$ such that 
$\psi^{-1}(S_k) \to S_k$ is a topologically locally trivial fibration. The subsets 
$S_k$ are called \textit{strata}.

Denote by $d_k:=\mathrm{dim}(\psi^{-1}(y_k))$ the dimension of the fiber of $ \psi $ 
over any point $y_k\in S_k$. The map $\psi$ is called \textit{small}, if
\begin{equation}\label{small}
\mathrm{codim}(S_k) =  \mathrm{dim}(Y) - \mathrm{dim}(S_k) > 2d_k
\end{equation}
 for every nondense stratum $ S_k $ for the map $\psi$. 

 Small maps play an important role in the calculation of the 
 \textit{intersection cohomology} groups of singular varieties \cite{GM}. 
 In general, the intersection cohomology $IH^*(Y)$ of an irreducible complex 
 projective $d$-dimensional variety $Y$ is a module over the singular 
 cohomology ring $H^*(Y)$, and satisfies Poincar\'e duality and the Hard 
 Lefschetz theorem \cite[Theorem 5.4.10]{BBD}. The latter means that for some 
 element $\omega\in H^2(Y)$ (the class of a hyperplane section) and for $0\leq 
 k\leq d$ the multiplication by $\omega^{d-k}$ mapping $IH^k(Y) \to 
 IH^{2d-k}(Y))$  is an isomorphism of vector spaces. One can define then,  for 
 $0\leq 
 k\leq d$,  the \textit{primitive intersection cohomology} of $Y$ 
 as $$IH_{\mathrm{prim}}^k(Y) =  IH^{2k}(Y)/ \omega\, IH^{2k-2}(Y).$$

The intersection cohomology of toric varieties is intimately related to the 
theory 
of convex polytopes \cite{dCMdec, Stanley}. We will now briefly review this relation.

\subsection{Intersection cohomology of toric varieties}
Let $P$ be a simplicial $d$-dimensional polytope, and denote by $f_i$ the 
number of its $i$-dimensional faces. One associates to $ P $ its 
\textit{$f$-polynomial} 
$$f(P,t) = t^d + f_0t^{d-1}+f_1t^{d-2}+ ... +f_{d-2}t + f_{d-1},$$  its 
\textit{$h$-polynomial} 
\begin{equation}\label{h-poly}
h(P,t) = h_dt^d+...+h_1t+h_0 := f(P,t-1),
\end{equation}
and its \textit{$g$-polynomial} 
\begin{equation}\label{g-poly} 
	g(P,t) = h_0+(h_1-h_0)t+(h_2-h_1)t^2+...+(h_{[d/2]}-h_{[d/2]-1})t^{[d/2]}.
\end{equation}

The following theorem calculates the dimension of the intersection cohomology groups of a toric orbifold.
\begin{proposition}\cite[\S5.2]{Fulton}\label{Poincar\'eh-poly}
Let $P$ be a simplicial $d$-dimensional rational polytope with $h$-polynomial $h(P, t) = \sum_{k=0}^d 
h_k t^k$, and let $X(\Sigma_P)$ be the associated toric orbifold. Then 
\begin{equation}\label{smoothhpol}
	\mathrm{dim}H^{2k}(X(\Sigma_P)) = h_k  \text{\,\,\, and \,\,\,} 
	H^{2k+1}(X(\Sigma_P)) = 0 \text{\,\,\, for \,\,\,} 0\leq k \leq d. 
\end{equation}
\end{proposition}

\begin{remark}
	Note that it follows from \eqref{smoothhpol} and the definition 
	\eqref{g-poly}, that the coefficient $h_k-h_{k+1}$ is the dimension of the 
	primitive cohomology of the toric variety $X(\Sigma_P)$ in degree $k$, and 
	thus, in particular, it is non-negative.
\end{remark}

When the polytope $P$ is not simplicial, the corresponding toric variety 
$X(\Sigma_P)$ has not only finite quotient singularities. In this case, the 
ordinary cohomology groups is not a purely combinatorial invariant, but depends 
also on some geometric data of the polytope \cite{McConnel}. A better invariant 
to consider is the intersection cohomology of $X(\Sigma_P)$.  
Then the "generalized" $h$-polynomial $h(P, t) = 
\sum\mathrm{dim}IH^{2k}(X(\Sigma_P))t^k$ is a purely combinatorial invariant, 
i.e. it can be again defined from the face lattice of the polytope $P$.  

For a simplicial polytope $P$, we have $H^*(X(\Sigma_P))=IH^*(X(\Sigma_P))$, 
and hence, in this case,  the generalized $h$-polynomial of $P$ coincides with 
the one defined in \eqref{h-poly}. Now, following Stanley \cite{Stanley}, we give combinatorial definitions of the 
$h$ and $g$ polynomials for a not necessarily simplicial polytope.

Let $P$ be a $d$-dimensional polytope and suppose that the $h$ and $g$ polynomials have been defined for all convex polytopes of dimension less than $d$. We set
\begin{equation}\label{h-polyARB}
h(P, t) = \sum_{F<P}g(F,  t)(t-1)^{d-1-\mathrm{dim}(F)},
\end{equation}
where the sum runs over all proper faces $F$ of $P$, including the empty face 
$\emptyset$, for which $g(\emptyset,t) = h(\emptyset,t) = 1$ and 
$\mathrm{dim}(\emptyset) = -1$. The polynomial $g(P,t)$ is defined from the 
polynomial $h(P,t)$ as in \eqref{g-poly}. Formulas \eqref{g-poly} and 
\eqref{h-polyARB} then inductively define the polynomials $ g $ and $ h $ for 
all polytopes.

\begin{remark}
Note that this definition agrees with  definitions given in  \eqref{h-poly} and \eqref{g-poly}, since the $g$-polynomial of any simplex equals $1$.
\end{remark}

The following theorem calculates the dimension of the intersection cohomology groups of a toric variety.

\begin{theorem}\cite{Fiesler}\label{IChandg}
Let $P$ be a $d$-dimensional rational polytope and let $X(\Sigma_P)$ be the toric variety, associated to the fan $\Sigma_P$. Then 
\begin{enumerate}[(i)]
\item $h(P, t) = \sum_{k=0}^{d} \mathrm{dim}IH^{2k}(X(\Sigma_P))t^k \text{\,\,\, and \,\,\,}  IH^{2k+1}(X(\Sigma_P))=0  \text{\,\, for\,\, } 0\leq k \leq d;$
\vspace{3pt}
\item $g(P,t) = \sum_{k=0}^{[d/2]} \mathrm{dim}IH_{\mathrm{prim}}^k(X(\Sigma_P))t^k.$
\end{enumerate}
\end{theorem}

Finally, we recall the following basic consequence of the \textit{decomposition 
theorem} \cite{BBD} applied to a small resolution of singularities of toric 
varieties.
\begin{theorem}\cite{Mauri}\label{dec}
 Let $X(\Sigma_{\widetilde{P}})$ be a $d$-dimensional simplicial toric variety 
 and let $\psi: X(\Sigma_{\widetilde{P}})\to X(\Sigma_P)$ be a small toric 
 morphism. Given a face $F<P$, we pick a point $y_F$ in the 
 $T_\mathfrak{t}$-orbit 
 $\mathcal{O}(\sigma)$, where $ \sigma $ is the cone over the face $F$ in the fan 
 $\Sigma_P$. Then
$$H^{2k}(\psi^{-1}(y_F)) \simeq IH_{\mathrm{prim}}^k(X(\Sigma_F)) \text{\,\, for\,\, }  0\leq k\leq [d/2].$$
In particular, for $F=P$, we have $y_P=0$ and $H^{2k}(\psi^{-1}(0)) \simeq IH_{\mathrm{prim}}^k(X(\Sigma_P)).$
\end{theorem}

\section{The $g$-polynomial of the Gale dual type-A root polytope}\label{sec:zhepol}
In this section, we study certain generalized Gale dual type-A root polytopes. 
The 
main result of this section, Theorem \ref{volosataya}, calculates the 
$g$-polynomials of these polytopes as generating functions of the number of a 
certain oriented graphs, graded by the number of their edges. 
Proposition \ref{recgpoly} presents a recursive formula, which is an effective 
tool for the calculation of these same $g$-polynomials.  
\subsection{The type-A root polytope}\label{S2.3}
The real vector space
$$\ak=\mathbb{R}^k/\mathbb{R}(1,...,1).$$
has a natural pairing with
$$\ak^*=\{(\varepsilon_1,...,\varepsilon_k)\in\mathbb{R}^k \, | \,\, 
\varepsilon_1+...+\varepsilon_k=0\},$$
where $(\varepsilon_1,...,\varepsilon_k)$ are the coordinates on 
$\mathbb{R}^k$.

Let $\Gamma_{\ak}^*$ be the integer lattice in the vector space $\ak^*$:
$$\Gamma_{\ak}^* = \{(\lambda_1,...,\lambda_k)\in\mathbb{Z}^k \, | \,\, 
\lambda_1+...+\lambda_k=0\}.$$
For $1\leq i\neq j\leq k$, we define the element $\alpha_{ij} = 
\varepsilon_i-\varepsilon_j \in\Gamma_{\ak}^*$, and we set
$$\Phi_k = \{\alpha_{ij} \, | \,\, 1\leq i\neq j\leq k\},$$ which is the set of 
roots of the $A_{k-1}$ root system. 

The \textit{type-A root polytope} is defined as the convex hull of the set of 
roots $\Phi_k$ in the vector space $\ak^*$. The Gale transformation converts 
the set $\Phi_k$ of $k(k-1)$ root 
vectors in $\ak^*$ into the set of $k(k-1)$ 
vectors in an appropriate $(k-1)^2$-dimensional space, which we will denote by $\Psi_k$. 
The \textit{Gale dual type-A root polytope} is the convex hull of the set of 
vectors from $\Psi_k$.

\begin{example}\label{ex:prism}
The Gale dual type-A root polytope, corresponding to the root system $A_2$, is a three-dimensional prism shown in Figure \ref{fig1}.
\end{example}
\begin{figure}[H]
\centering
\begin{tikzpicture}[scale=1]
\path [draw=black, thick] (0,2) edge (3,2);
\path [draw=black, thick] (3,2) edge (1,1.2);
\path [draw=black, thick] (0,2) edge (1,1.2);
\path [draw=black, thick, dashed] (0,0.2) edge (3,0.2);
\path [draw=black, thick] (0,0.2) edge (1,-0.6);
\path [draw=black, thick] (1,-0.6) edge (3,0.2);
\path [draw=black, thick] (0,2) edge (0,0.2);
\path [draw=black, thick] (3,2) edge (3,0.2);
\path [draw=black, thick] (1,1.2) edge (1, -0.6);
\draw [fill] (1,-0.6) circle [radius=0.02];
\node [below] at (1,-0.55) {\tiny $\beta_{12}$};
\draw [fill] (0,0.2) circle [radius=0.02];
\node [below] at (0,0.2) {\tiny $\beta_{31}$};
\draw [fill] (3,0.2) circle [radius=0.02];
\node [below] at (3,0.2) {\tiny $\beta_{23}$};
\draw [fill] (1,1.2) circle [radius=0.02];
\node [right] at (0.9,1.1) {\tiny $\beta_{21}$};
\draw [fill] (0,2) circle [radius=0.02];
\node [left] at (0,2) {\tiny $\beta_{13}$};
\draw [fill] (3,2) circle [radius=0.02];
\node [right] at (3,2) {\tiny $\beta_{32}$};
\end{tikzpicture}
\vskip 14pt
\setlength{\belowcaptionskip}{-8pt}\caption{The Gale dual root polytope, corresponding to the root system $A_2$.}\label{fig1}
\end{figure}
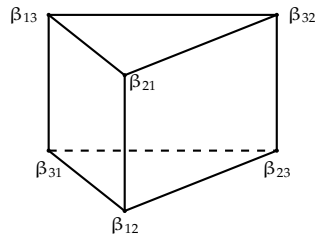

\begin{remark}
The Gale dual type-A polytope is a particular case of a \textit{Lawrence 
polytope}  \cite{HSt}.
\end{remark}

In this paper, we will study the Gale dual root polytope, obtained from the type-A root system with multiplicities. For this, we fix integers  $r_{ij}\in\mathbb{N}$ for $ 1\leq i< j\leq k$, set $r_{ji}=r_{ij}$ and consider the ordered sequence 
\begin{equation}\label{A}
\A(r_{12},r_{13},...,r_{k-1\,k}) = 
[\alpha_{12},...,\alpha_{12},\alpha_{21},...,\alpha_{21},\alpha_{13},...,\alpha_{13},...,\alpha_{k\,
 k-1}]\in\Gamma_{\ak}^*
\end{equation}
of the root vectors $\alpha_{ij}\in\Phi_k$, where the vector $\alpha_{ij}$ is 
repeated $r_{ij}$ times. Following the notation from page \pageref{seq:toric}, 
we set 
$$n:=|\A| = 2\sum_{1\leq i< j\leq k}r_{ij} \text{\,\, and \,\,} d: = n-k+1.$$
We consider the Gale dual sequence $\B(r_{12},...,r_{k-1 k})$ and denote by 
$\Pi(r_{12},...,r_{k-1 k})$ the polytope obtained as the convex hull of the 
vectors in $\B(r_{12},...,r_{k-1 k})$. It follows from Lemma \ref{Galebas}, 
that similarly to the root configuration, this dual configuration is not full 
dimensional: $\Pi(r_{12},...,r_{k-1 k})$ is a $d-1$-dimensional polytope in a 
$d$-dimensional vector space.

We can also deduce from Lemma \ref{Galebas} that the polytope $\Pi(r_{12},...,r_{k-1 k})$ does not contain the origin in its interior;
in fact, the toric variety $X(\A(r_{12},...,r_{k-1 k}), 0)$ (cf. page \pageref{X_0def}) is an affine cone over the singular projective toric variety $X(\Sigma_{\Pi(r_{12},...,r_{k-1 k})})$ (cf. Proposition \ref{fanforGIT}).

\noindent\textbf{Notation:} The fan corresponding to the toric variety $X(\A(r_{12},...,r_{k-1 k}), 0)$ consists of the cones over the faces of $\Pi$, including $\Pi$ itself; we denote the cone over the face $F<\Pi$ by $\cone(F)$.

\begin{example}\label{k=2space}
A simple calculation shows that for $k=2$  the polytope $\Pi(r_{12})$ is a product of two $(r_{12}-1)$-dimensional simplices, and thus the toric variety $X(\A(r_{12}), 0)$ is an affine cone over the product of projective spaces $\mathbb{P}^{r_{12}-1}\times\mathbb{P}^{r_{12}-1}$.
\end{example}

\subsection{Small maps and the combinatorics of the $g  $-polynomial}\label{S4.2}
Before we  formulate the main result of this section, Theorem \ref{volosataya}, 
we introduce some extra notation related to graphs. 
\begin{itemize}
\item We denote by $\mathbb{K}_k=\mathbb{K}_k(r_{12},...,r_{k-1 k})$ the directed graph with vertex set $\{1, 
2,...,k\}$ and with number of oriented edges from $i$ to $j$ equal to 
$r_{ij}$.  
 We will use the notation $\overleftarrow{ij}$ for the edge directed from $j$ 
 to $i$. 

\item Denote by  $\mathbb{G}_k=\mathbb{G}_k(r_{12},...,r_{k-1 k})$ the graph obtained from $\mathbb{K}_k$  by 
deleting all edges $\overleftarrow{i1}$ for $1<i\leq k$. To emphasize the break 
of 
symmetry: we will color the first vertex in $\mathbb{G}_k$ in red,
and the other vertices in black.

\item A directed graph is \textit{acyclic} if it has no directed cycles. 

\item A directed graph is \textit{naked} if every edge of this graph is contained in at least one directed cycle.

\item We will say that a directed graph $G$ is rooted at the $i^{\text{th}}$ 
vertex if there is a directed path in $G$ from any vertex to the 
$i^{\text{th}}$ vertex. 
\end{itemize}

\begin{theorem}\label{volosataya} Let $\Pi = \Pi(r_{12},...,r_{k-1 k})$ be the Gale dual type-A root polytope defined above, and let $\hat{d}=\sum_{i<j} r_{ij} - k+1$. Then 
$$g(\Pi, t+1) = 
\hat{g}_{0}+t\hat{g}_1+t^2\hat{g}_{2}+...+t^{\hat{d}}\hat{g}_{\hat{d}},$$ where 
$\hat{g}_{i}$ counts the number of acyclic subgraphs $G\subset \mathbb{G}_k$ with $k$ vertices and $k-1+i$ edges, which are rooted at the first vertex.
\end{theorem}
\begin{example} 
For $\Pi = \Pi(1,1,1)$ from Example \ref{ex:prism} we have $\hat{g}_0=3$ and $\hat{g}_1=2$ (cf. Figure \ref{fig:g-poly}), hence $g(\Pi, t) = 3+2(t-1)=1+2t$. 
\end{example}
\begin{figure}[H]
\begin{minipage}{.19\textwidth}
\centering
\begin{tikzpicture}[scale=1.1]
\path [bend right, draw=black, thick] (1,0.6) edge node[currarrow, pos=0.5, xscale=-1, sloped] {} (0,0);
\path [bend right, draw=black, thick] (1,-0.6) edge node[currarrow, pos=0.5, sloped] {} (1,0.6);
\draw [red, fill] (0,0) circle [radius=0.05];
\draw [fill] (1,0.6) circle [radius=0.05];
\draw [fill] (1,-0.6) circle [radius=0.05];
\node [above] at (0,0) {\tiny $1$};
\node [right] at (1,0.6) {\tiny $2$};
\node [right] at (1,-0.6) {\tiny $3$};
\end{tikzpicture}
\end{minipage}
\begin{minipage}{.19\textwidth}
\centering
\begin{tikzpicture}[scale=1.1]
\path [bend left, draw=black, thick] (1,-0.6) edge node[currarrow, pos=0.5, xscale=-1, sloped] {} (0,0);
\path [bend left, draw=black, thick] (1,0.6) edge node[currarrow, pos=0.5, sloped] {} (1,-0.6);
\draw [red, fill] (0,0) circle [radius=0.05];
\draw [fill] (1,0.6) circle [radius=0.05];
\draw [fill] (1,-0.6) circle [radius=0.05];
\node [above] at (0,0) {\tiny $1$};
\node [right] at (1,0.6) {\tiny $2$};
\node [right] at (1,-0.6) {\tiny $3$};
\end{tikzpicture}
\end{minipage}
\begin{minipage}{.19\textwidth}
\centering
\begin{tikzpicture}[scale=1.1]
\path [bend right, draw=black, thick] (1,0.6) edge node[currarrow, pos=0.5, xscale=-1, sloped] {} (0,0);
\path [bend left, draw=black, thick] (1,-0.6) edge node[currarrow, pos=0.5, xscale=-1, sloped] {} (0,0);
\draw [red, fill] (0,0) circle [radius=0.05];
\draw [fill] (1,0.6) circle [radius=0.05];
\draw [fill] (1,-0.6) circle [radius=0.05];
\node [above] at (0,0) {\tiny $1$};
\node [right] at (1,0.6) {\tiny $2$};
\node [right] at (1,-0.6) {\tiny $3$};
\end{tikzpicture}
\end{minipage}
\begin{minipage}{.19\textwidth}
\centering
\begin{tikzpicture}[scale=1.1]
\path [bend right, draw=black, thick] (1,0.6) edge node[currarrow, pos=0.5, xscale=-1, sloped] {} (0,0);
\path [bend left, draw=black, thick] (1,-0.6) edge node[currarrow, pos=0.5, xscale=-1, sloped] {} (0,0);
\path [bend right, draw=black, thick] (1,-0.6) edge node[currarrow, pos=0.5, sloped] {} (1,0.6);
\draw [red, fill] (0,0) circle [radius=0.05];
\draw [fill] (1,0.6) circle [radius=0.05];
\draw [fill] (1,-0.6) circle [radius=0.05];
\node [above] at (0,0) {\tiny $1$};
\node [right] at (1,0.6) {\tiny $2$};
\node [right] at (1,-0.6) {\tiny $3$};
\end{tikzpicture}
\end{minipage}
\begin{minipage}{.19\textwidth}
\centering
\begin{tikzpicture}[scale=1.1]
\path [bend right, draw=black, thick] (1,0.6) edge node[currarrow, pos=0.5, xscale=-1, sloped] {} (0,0);
\path [bend left, draw=black, thick] (1,-0.6) edge node[currarrow, pos=0.5, xscale=-1, sloped] {} (0,0);
\path [bend left, draw=black, thick] (1,0.6) edge node[currarrow, pos=0.5, sloped] {} (1,-0.6);
\draw [red, fill] (0,0) circle [radius=0.05];
\draw [fill] (1,0.6) circle [radius=0.05];
\draw [fill] (1,-0.6) circle [radius=0.05];
\node [above] at (0,0) {\tiny $1$};
\node [right] at (1,0.6) {\tiny $2$};
\node [right] at (1,-0.6) {\tiny $3$};
\end{tikzpicture}
\end{minipage}
\vskip 14pt
\setlength{\belowcaptionskip}{-8pt}\caption{Acyclic subgraphs of $\mathbb{G}_3$ that contain a spanning tree rooted at the first vertex: there are $3$ graphs with two edges and $2$ graphs with three edges.}\label{fig:g-poly}
\end{figure}
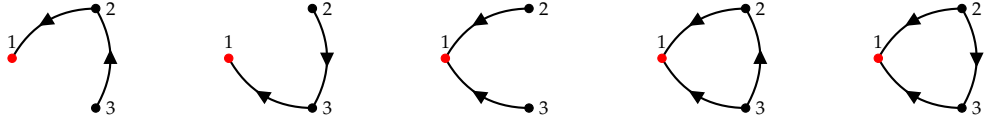

The strategy of the proof of Theorem \ref{volosataya} is as follows. First, we construct a toric resolution of the singularities of the affine variety $X(\A(r_{12},...,r_{k-1 k}), 0)$, and
introduce some graph-theoretic tools based on the fact that weights of the 
$T_\mathfrak{a}$-action correspond to edges of an oriented complete graph. 
Next, we prove that this resolution is \textit{small} (cf. \S\ref{S3.1}). The  
Decomposition 
Theorem then allows us 
to identify the intersection cohomology of $X(\A(r_{12},...,r_{k-1 k}), 0)$ 
with the cohomology of the fibers of this small resolution. Finally, we 
describe the cohomology of the fibers using generating functions for the number 
of certain oriented graphs.

\begin{lemma}\cite[Lemma 4.12]{HW}\label{cham_1}
The simplicial cone $\cone(\{\alpha_{12},\alpha_{13},...,\alpha_{1k}\})$ is a 
chamber in $Ch(\A(r_{12},...,r_{k-1 k}))$.
\end{lemma}

Note that $\theta_1 = (k-1,-1,...,-1)\in 
\cone(\{\alpha_{12},\alpha_{13},...,\alpha_{1k}\})$,  and consider the 
corresponding toric variety
$X(\A(r_{12},...,r_{k-1 k}), \theta_1)$. As our fan is unimodular, this variety is smooth.

\begin{example}
Below is the chamber complex for the root system $A_2$ and the triangulation of the Gale dual root polytope $\Pi(1,1,1)$ (cf. Example \ref{ex:prism}) given by the chamber $\{\alpha_{12}, \alpha_{13}\}$.
\end{example}

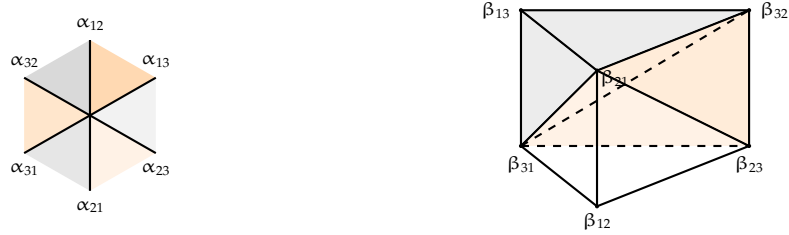
\begin{figure}[H]
\begin{minipage}{.49\textwidth}
\centering
\begin{tikzpicture}[scale=1]
\fill [fill=orange!30, ] (0,1) -- (0,0) -- ({sqrt(3/4)},0.5);
\fill [fill=gray!30, ] (0,1) -- (0,0) -- (-{sqrt(3/4)},0.5);
\fill [fill=gray!10, ] ({sqrt(3/4)},0.5) -- (0,0) -- ({sqrt(3/4)},-0.5);
\fill [fill=orange!10, ] ({sqrt(3/4)},-0.5) -- (0,0) -- (0,-1);
\fill [fill=gray!20, ] (-{sqrt(3/4)},-0.5) -- (0,0) -- (0,-1);
\fill [fill=orange!20, ] (-{sqrt(3/4)},-0.5) -- (0,0) -- (-{sqrt(3/4)},0.5);
\path [draw=black, thick]  (-{sqrt(3/4)},-0.5) edge  ({sqrt(3/4)},0.5);
\path [draw=black, thick]  ({sqrt(3/4)},-0.5) edge  (-{sqrt(3/4)},0.5);
\path [draw=black, thick]  (0,-1) edge  (0,1);
\draw [fill] (0,0) circle [radius=0.02];
\node [above] at (0,1) {\tiny $\alpha_{12}$};
\node [below] at (0,-1) {\tiny $\alpha_{21}$};
\node [above] at ({sqrt(3/4)},0.5) {\tiny $\alpha_{13}$};
\node [below] at ({sqrt(3/4)},-0.5) {\tiny $\alpha_{23}$};
\node [above] at (-{sqrt(3/4)},0.5) {\tiny $\alpha_{32}$};
\node [below] at (-{sqrt(3/4)},-0.5) {\tiny $\alpha_{31}$};
\end{tikzpicture}
\end{minipage}
\begin{minipage}{.49\textwidth}
\centering
\begin{tikzpicture}[scale=1]
\draw [ fill=lightgray!30] (0,0.2) -- (1,1.2) -- (0,2);
\draw [ fill=lightgray!30] (3,2) -- (1,1.2) -- (0,2);
\draw [ fill=lightgray!30] (3,2) -- (1,1.2) -- (0,0.2);
\draw [ fill=orange!15] (3,0.2) -- (1,1.2) -- (3,2);
\draw [ fill=orange!10] (3,0.2) -- (1,1.2) -- (0,0.2);
\path [draw=black, thick] (0,2) edge (3,2);
\path [draw=black, thick] (3,2) edge (1,1.2);
\path [draw=black, thick] (0,2) edge (1,1.2);
\path [draw=black, thick, dashed] (0,0.2) edge (3,0.2);
\path [draw=black, thick] (0,0.2) edge (1,-0.6);
\path [draw=black, thick] (1,-0.6) edge (3,0.2);
\path [draw=black, thick] (0,2) edge (0,0.2);
\path [draw=black, thick] (3,2) edge (3,0.2);
\path [draw=black, thick] (1,1.2) edge (1, -0.6);
\path [draw=orange, thick] (1,1.2) edge (3,0.2);
\path [draw=orange, thick] (0,0.2) edge (1,1.2);
\path [draw=orange, thick, dashed] (0,0.2) edge (3,2);
\draw [fill] (1,-0.6) circle [radius=0.02];
\node [below] at (1,-0.55) {\tiny $\beta_{12}$};
\draw [fill] (0,0.2) circle [radius=0.02];
\node [below] at (0,0.2) {\tiny $\beta_{31}$};
\draw [fill] (3,0.2) circle [radius=0.02];
\node [below] at (3,0.2) {\tiny $\beta_{23}$};
\draw [fill] (1,1.2) circle [radius=0.02];
\node [right] at (0.9,1.1) {\tiny $\beta_{21}$};
\draw [fill] (0,2) circle [radius=0.02];
\node [left] at (0,2) {\tiny $\beta_{13}$};
\draw [fill] (3,2) circle [radius=0.02];
\node [right] at (3,2) {\tiny $\beta_{32}$};
\end{tikzpicture}
\end{minipage}
\vskip 14pt
\setlength{\belowcaptionskip}{-8pt}\caption{The chamber complex for the root system $A_2$ and the triangulation of the Gale dual root polytope given by the chamber $\{\alpha_{12}, \alpha_{13}\}$.}\label{fig:ch(A)}
\end{figure}
To simplify the notation, from now on, we omit the dependence on $m, 
r_{12},...,r_{k-1 k}$:  
throughout this chapter, we will use the notation
\begin{itemize}\label{omitnot}
\item $\A:=\A(r_{12},...,r_{k-1 k})$;
\item $\Pi := \Pi(r_{12},...,r_{k-1 k})$;
\item $X:=X(\A(r_{12},...,r_{k-1 k}), 0)$;
\item  $ \widetilde{X}:= X(\A(r_{12},...,r_{k-1 k}), \theta_1)$ where $\theta_1= (k-1,-1,...,-1)$;
\item $\Sigma$ and $\widetilde{\Sigma}$ for the toric fans of $X$ and $\widetilde{X}$, respectively;
\item $\varphi:  \widetilde{X} \to X$ for the canonical morphism defined in Remark \ref{canonocalmor}; clearly, $\varphi$ is a toric morphism compatible with the refinement of fans $\widehat{\varphi}_\Sigma: \tilS\to \Sigma$.
\end{itemize}

It is well known (cf. e.g. \cite{Yau}) that the fiber of the toric morphism 
$\varphi:  \widetilde{X} \to X$ over a point $x\in X$ depends only on the orbit 
$\mathcal{O}(\sigma)\subset X$ (cf. page \pageref{notation:toricorb}) that contains $x$. We thus note  that the decomposition $$X = \bigsqcup_{F<\Pi} 
\mathcal{O}(\cone(F)),$$ where $F$ runs over the faces of $\Pi$, including  the 
empty face and $\Pi$ itself, is a stratification (cf. \S\ref{S3.1}) for the 
canonical morphism $\varphi:\widetilde{X}\to X$.

Now we are ready to formulate our main technical result, whose proof will be the contect of the following paragraph.

\begin{theorem}\label{smallmor}
The canonical morphism $\varphi:  \widetilde{X} \to X$ is small (cf. 
\S\ref{S3.1}). 
\end{theorem} 

\subsection{Proof of Theorem \ref{smallmor}}

Clearly, we have $ \dim \widetilde{X}=\dim X$, thus to prove that the map 
$\varphi$ is small, we need to show that for any non-empty face $F<\Pi$ 
\begin{equation}\label{oursmallest} 
\mathrm{codim}(\mathcal{O}(\cone(F))\subset X) > 2\mathrm{dim}(\varphi^{-1}(y_F)), 
\end{equation}
where $y_F$ is a point in the orbit $\mathcal{O}(\cone(F))$, which is a stratum 
for our stratification. We start with the calculation of the codimension of 
$\mathcal{O}(\cone(F))$ in $X$.

Associating to each element $\alpha_{ij}\in \A $  the edge $\overleftarrow{ij}$ of the graph $\mathbb{K}_k$, we obtain a natural correspondence  between subsequences
$\mathcal{A}\subset\A$ and subgraphs of $\mathbb{K}_k$.  Applying Gale 
duality, we also obtain the correspondence 
between subgraphs of $\mathbb{K}_k$ and subsets of rays of the fan $\Sigma$ 
(and thus rays of the fan $\widetilde{\Sigma}$ as well). \label{corresp}

The following lemma describes the faces of the polytope $\Pi$ in terms of subgraphs in $\mathbb{K}_k$. 
\begin{lemma}\label{facegraphs}
Under the correspondence described above, naked graphs correspond to  faces of the polytope $\Pi$. 
\end{lemma}
\begin{proof}
For any subsequence $\mathcal{A}\subset\A$,  
we denote by $G_\mathcal{A}$ the graph 
corresponding to $\mathcal{A}$. 
By Gale duality, faces of $\Pi$ correspond to subsequences 
$\mathcal{A}\subset\A$  that have a positive linear combination summing up 
to zero:
\begin{equation}\label{sumalpha}
 \sum_{\alpha\in \mathcal{A}} m_{\alpha}\alpha = 0\quad\text{ with 
 }m_{\alpha}>0.
\end{equation}
Clearly, for any subsequence $\mathcal{C}\subset \mathcal{A}$ that corresponds 
to a directed cycle $G_\mathcal{C}$  of $G_\mathcal{A}$ we have 
$\sum_{\alpha\in \mathcal{C}} \alpha=0$ and one can subtract a multiple of this 
sum from \eqref{sumalpha} to obtain a proper subsequence 
$\mathcal{A}'\subset\mathcal{A}$ that also has a positive linear combination summing up to zero.

Assume that $G_\mathcal{A}$ is not a naked graph; then, repeating the procedure 
described above, we arrive at a nonempty subsequence of vectors 
$\bar{\mathcal{A}}\subset\A$ that has a positive linear combination summing up 
to zero, 
and such that the corresponding 
directed graph $G_{\bar{\mathcal{A}}}\subset \mathbb{K}_k$ is acyclic. Since 
any directed acyclic graph contains a vertex with only out-edges, we arrive at 
a  contradiction with \eqref{sumalpha}. 
\end{proof}
\noindent\textbf{Notation}\label{not:naked} We will denote the naked graph on  
$k$ vertices, corresponding to the face $F<\Pi$ by $G_F$.  In particular, the naked graph $G_\Pi$ has no edges and $G_\emptyset=\mathbb{K}_k$.

\begin{example}
The naked graphs from Figure \ref{fig:naked} correspond to the following faces of the prism $\Pi(1,1,1)$ (cf. Figure \ref{fig1}): the two-dimensional simplex $\{\beta_{13}, \beta_{32}, \beta_{21}\}$, the square $\{\beta_{13}, \beta_{31}, \beta_{32}, \beta_{23}\}$, two edges $\{\beta_{13}, \beta_{31}\}$ and $\{\beta_{13}, \beta_{23}\}$, the vertex $\{\beta_{13}\}$.
\end{example}
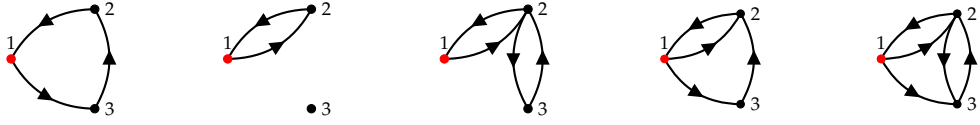
\begin{figure}[H]
\begin{minipage}{.19\textwidth}
\centering
\begin{tikzpicture}[scale=1.1]
\path [bend right, draw=black, thick] (1,0.6) edge node[currarrow, pos=0.5, xscale=-1, sloped] {} (0,0);
\path [bend right, draw=black, thick] (1,-0.6) edge node[currarrow, pos=0.5, sloped] {} (1,0.6);
\path [bend left, draw=black, thick] (1,-0.6) edge node[currarrow, pos=0.5, xscale=1, sloped] {} (0,0);
\draw [red, fill] (0,0) circle [radius=0.05];
\draw [fill] (1,0.6) circle [radius=0.05];
\draw [fill] (1,-0.6) circle [radius=0.05];
\node [above] at (0,0) {\tiny $1$};
\node [right] at (1,0.6) {\tiny $2$};
\node [right] at (1,-0.6) {\tiny $3$};
\end{tikzpicture}
\end{minipage}
\begin{minipage}{.19\textwidth}
\centering
\begin{tikzpicture}[scale=1.1]
\path [bend right, draw=black, thick] (1,0.6) edge node[currarrow, pos=0.5, xscale=-1, sloped] {} (0,0);
\path [bend left, draw=black, thick] (1,0.6) edge node[currarrow, pos=0.5, xscale=1, sloped] {} (0,0);
\draw [red, fill] (0,0) circle [radius=0.05];
\draw [fill] (1,0.6) circle [radius=0.05];
\draw [fill] (1,-0.6) circle [radius=0.05];
\node [above] at (0,0) {\tiny $1$};
\node [right] at (1,0.6) {\tiny $2$};
\node [right] at (1,-0.6) {\tiny $3$};
\end{tikzpicture}
\end{minipage}
\begin{minipage}{.19\textwidth}
\centering
\begin{tikzpicture}[scale=1.1]
\path [bend left, draw=black, thick] (1,0.6) edge node[currarrow, pos=0.5, xscale=-1, sloped] {} (1,-0.6);
\path [bend right, draw=black, thick] (1,0.6) edge node[currarrow, pos=0.5, xscale=1, sloped] {} (1,-0.6);
\path [bend right, draw=black, thick] (1,0.6) edge node[currarrow, pos=0.5, xscale=-1,  sloped] {} (0,0);
\path [bend left, draw=black, thick] (1,0.6) edge node[currarrow, pos=0.5, xscale=1, sloped] {} (0,0);
\draw [red, fill] (0,0) circle [radius=0.05];
\draw [fill] (1,0.6) circle [radius=0.05];
\draw [fill] (1,-0.6) circle [radius=0.05];
\node [above] at (0,0) {\tiny $1$};
\node [right] at (1,0.6) {\tiny $2$};
\node [right] at (1,-0.6) {\tiny $3$};
\end{tikzpicture}
\end{minipage}
\begin{minipage}{.19\textwidth}
\centering
\begin{tikzpicture}[scale=1]
\path [bend right, draw=black, thick] (1,0.6) edge node[currarrow, pos=0.5, xscale=-1, sloped] {} (0,0);
\path [bend left, draw=black, thick] (1,0.6) edge node[currarrow, pos=0.5, xscale=1, sloped] {} (0,0);
\path [bend left, draw=black, thick] (1,-0.6) edge node[currarrow, pos=0.5, xscale=1, sloped] {} (0,0);
\path [bend right, draw=black, thick] (1,-0.6) edge node[currarrow, pos=0.5, sloped] {} (1,0.6);
\draw [red, fill] (0,0) circle [radius=0.05];
\draw [fill] (1,0.6) circle [radius=0.05];
\draw [fill] (1,-0.6) circle [radius=0.05];
\node [above] at (0,0) {\tiny $1$};
\node [right] at (1,0.6) {\tiny $2$};
\node [right] at (1,-0.6) {\tiny $3$};
\end{tikzpicture}
\end{minipage}
\begin{minipage}{.19\textwidth}
\centering
\begin{tikzpicture}[scale=1]
\path [bend left, draw=black, thick] (1,0.6) edge node[currarrow, pos=0.5, xscale=-1, sloped] {} (1,-0.6);
\path [bend right, draw=black, thick] (1,0.6) edge node[currarrow, pos=0.5, xscale=1, sloped] {} (1,-0.6);
\path [bend right, draw=black, thick] (1,0.6) edge node[currarrow, pos=0.5, xscale=-1,  sloped] {} (0,0);
\path [bend left, draw=black, thick] (1,0.6) edge node[currarrow, pos=0.5, xscale=1, sloped] {} (0,0);
\path [bend left, draw=black, thick] (1,-0.6) edge node[currarrow, pos=0.5, xscale=1,  sloped] {} (0,0);
\draw [red, fill] (0,0) circle [radius=0.05];
\draw [red, fill] (0,0) circle [radius=0.05];
\draw [fill] (1,0.6) circle [radius=0.05];
\draw [fill] (1,-0.6) circle [radius=0.05];
\node [above] at (0,0) {\tiny $1$};
\node [right] at (1,0.6) {\tiny $2$};
\node [right] at (1,-0.6) {\tiny $3$};
\end{tikzpicture}
\end{minipage}
\vskip 14pt
\setlength{\belowcaptionskip}{-8pt}\caption{Naked graphs on 3 vertices.}\label{fig:naked}
\end{figure}

\begin{lemma}\label{lemma:trees}
Denote by $\mathfrak{c}_{\textcolor{red}{1}}$  the chamber 
$\cone(\{\alpha_{12},\alpha_{13},...,\alpha_{1k}\})$ in $ Ch(\A)$ (cf. Lemma 
\ref{cham_1}). Then the cones $\cone(\alpha^I)$ for $I\in\mathrm{BInd}(\A, 
\mathfrak{c}_{\textcolor{red}{1}})$ correspond to the spanning trees in 
$\mathbb{G}_k$, rooted at the first vertex. 
\end{lemma}
\begin{proof}
First, note that the cone $\cone(\alpha^I)$ is simplicial if and only if the 
vectors $ \{\alpha_i\}_{i\in I} $ are linearly independent; this happens if and 
only if the corresponding  graph has no (undirected) cycles, i.e. it  is a 
tree. 

Next, note that $\theta_1= (k-1,-1,...,-1)$ is a point in the interior of $\mathfrak{c}_{\textcolor{red}{1}}$, hence $\theta_1$ lies in the interior of $\cone(\alpha^I)$ for each $I\in\mathrm{BInd}(\A, \mathfrak{c}_{\textcolor{red}{1}})$. In other words, for each $I\in\mathrm{BInd}(\A, \mathfrak{c}_{\textcolor{red}{1}})$,  $\theta_1$ can be written as a linear combination of vectors $\alpha_i$, $i\in I$ with positive coefficients; this implies that the corresponding graph has at least one out-edge from the vertices $2, 3,..., k$. 
Clearly, such graph cannot have an out-edge starting at the first vertex, since this would create a cycle, hence the statement follows. 
\end{proof}

Since the fan $\tilS$ is simplicial, by Gale duality, we obtain the following 
statement. 
\begin{corollary}\label{conedimcount}
For $0\leq l  \leq (2\sum_{i<j}r_{ij} - k +1)$, there is a bijection between 
the codimension-$l$ cones in the fan $\widetilde{\Sigma}$ and the connected  
subgraphs of $\mathbb{K}_k$ with $ k $ vertices and $k-1+l$ edges, rooted at 
the first vertex.
\end{corollary}	

For any graph $G$, we introduce the notation $e(G)$ and $s(G)$ for the number of edges and connected components of $G$, correspondingly. It follows from Corollary \ref{conedimcount} that 
\begin{equation}\label{codimcone}
\mathrm{codim}(\cone(F)) = e(G_F)+s(G_F)-k,
\end{equation}
and thus
\begin{equation}\label{codimstrata}
\mathrm{codim}(\mathcal{O}(\cone(F))\subset X) =\mathrm{dim}(\cone(F)) =  2\sum_{1\leq i< j\leq k}r_{ij} -e(G_F)-s(G_F)+1.
\end{equation}

Since our goal is to compare this number with the right-hand side of \eqref{oursmallest}, our next step will be to analyze the fibers of $\varphi$. 

In general, the combinatorics of the fan refinement provides an explicit description of the fibers of any toric morphism (cf. \cite{Yau}).
To prove Theorems \ref{volosataya} and \ref{smallmor} we will only need to 
calculate the Betti numbers of these fibers (see Theorem \ref{Poincar\'eFib} for 
the 
result).

\begin{lemma}\label{fibergraphs} 
Let $F$ be a face of $\Pi$, and let $G_F$ be the corresponding naked graph (cf. Lemma \ref{facegraphs}).
Denote by $\mathcal{G}^{\textcolor{red}{1}}_F$  the set of connected  
subgraphs of $\mathbb{K}_k$ on $ k $ vertices, which have the same directed 
cycles as the naked 
graph $G_F$, and which are rooted at the first vertex. 
Then, under the morphism $\widehat{\varphi}_\Sigma$, the cone $\sigma\in\tilS$ 
maps to the cone $\cone(F)\in\Sigma$  if and only if $\sigma$ corresponds (cf. 
page \pageref{corresp}) to some graph in the set 
$\mathcal{G}^{\textcolor{red}{1}}_F$.
\end{lemma}
\begin{proof}
Denote by $G_\sigma$ the subgraph of $\mathbb{K}_k$ that corresponds to the cone $\sigma\subset \widetilde{\Sigma}$ under the correspondence described on page \pageref{corresp}. Repeating the argument from the proof of Lemma \ref{lemma:trees}, we conclude that $G_\sigma$ is a connected graph rooted at the first vertex. 

Recall that $\widehat{\varphi}_\Sigma(\sigma) = \cone(F)$ if and only if 
$\sigma \subset \cone(F)$
and $\cone(F)$ is the minimal cone in $\Sigma$ that contains $\sigma$.
Note that the first condition is equivalent to the fact that $G_\sigma$ contains the naked graph $G_F$, while the second condition requires $G_F$ to be the maximal naked graph contained in $G_\sigma$; thus $G_\sigma$ should have the same directed cycles as $G_F$. 
\end{proof}

\begin{theorem}\cite[Corollary 4.7]{dCMMtor}\label{Poincar\'eFib}
Let $\varphi: \widetilde{X}\to X$ be as on page \pageref{omitnot}. For any face $F<\Pi$, let 
$d_l(F)$ be  the number of $(\mathrm{dim}(\cone(F))-l)$-dimensional cones 
$\sigma\in\tilS$, such that the (relative) interior of $\sigma$ maps to the (relative) interior of 
the cone $\cone(F)\in\Sigma$:
$$d_l(F) = |\{\sigma\in\tilS \, | \,\,\, 
\widehat{\varphi}_\Sigma(\sigma)=\cone(F), \,\, 
\mathrm{codim}(\sigma)-\mathrm{codim}(\cone(F))=l\}|.$$ 
Then, for any point  $y_F$ in $\mathcal{O}(\cone(F))\subset X$, the Poincar\'e 
polynomial of the fiber $\varphi^{-1}(y_F)$ has the following form
$$\sum_{l\geq 0} \mathrm{dim}H^{2l}(\varphi^{-1}(y_F))t^{2l} = \sum_{l\geq 0} d_l(F)(t^2-1)^l. $$
\end{theorem}
Putting together Lemma \ref{fibergraphs} and equation \eqref{codimcone}, we arrive at the following interpretation of the numbers $d_l(F)$ in terms of graphs. 
\begin{lemma}\label{dlgraphs}
The number $d_l(F)$ defined in Theorem \ref{Poincar\'eFib} is equal to the number of graphs in $\mathcal{G}^{\textcolor{red}{1}}_F$ with $e(G_F)+s(G_F)-1+l$ edges. 
\end{lemma}
In particular, it follows from Theorem \ref{Poincar\'eFib} that for $y_F\in\mathcal{O}(\cone(F))$ the dimension of the fiber $\varphi^{-1}(y_F)$ is bounded from above by 
$$ \displaystyle{\max_{G\in\mathcal{G}^{\textcolor{red}{1}}_F}} \{e(G)\}-e(G_F)-s(G_F)+1. $$
Hence using expression \eqref{codimstrata} for the codimension of the strata
 $\mathcal{O}(\cone(F))$ in $X$, we can conclude that to prove 
\eqref{oursmallest}, it is enough to show that
\begin{equation} 
2\sum_{1\leq i< j\leq k}r_{ij} > \displaystyle{2\max_{G\in\mathcal{G}^{\textcolor{red}{1}}_F}} \{e(G)\} - e(G_F)-s(G_F)+1 
\end{equation}
for any non-empty $F<\Pi$. Note that we can think of $$\displaystyle{2\max_{G\in\mathcal{G}^{\textcolor{red}{1}}_F}} \{e(G)\} - e(G_F)$$ as of the number of edges in the graph $G\subset\mathbb{K}_k$, such that $G$ is obtained from the naked graph $G_F$ by adding all possible edges between the connected components of $G_F$. Then we have trivially
$$\displaystyle{2\max_{G\in\mathcal{G}^{\textcolor{red}{1}}_F}} \{e(G)\} - 
e(G_F) \leq e(\mathbb{K}_k) = 2\sum_{1\leq i< j\leq k}r_{ij},$$ with equality 
if and only if $G_F$ is a disjoint union of complete graphs. Hence we obtain 
the inequality  
$$\displaystyle{2\max_{G\in\mathcal{G}^{\textcolor{red}{1}}_F}} \{e(G)\} - e(G_F)-s(G_F)+1 \leq 2\sum_{1\leq i< j\leq k}r_{ij} - s(G_F)+1 \leq 2\sum_{1\leq i< j\leq k}r_{ij},$$
with equality if and only if $G_F$ is a disjoint union of complete graphs and 
$s(G_F)=1$. In this latter case, $G_F$ is connected and complete, and thus 
$G_F=\mathbb{K}_k$. This corresponds to the empty face $F=\emptyset$ with  
$\mathrm{codim}(\mathcal{O}(\cone(\emptyset))) = 0$, and Theorem \ref{smallmor} 
follows. 

\subsection{Proof of Theorem \ref{volosataya}}\label{S4.4}

Now we can finish the proof of Theorem \ref{volosataya}. 

Since the map $\varphi: \widetilde{X}\to X$ is small, it follows from Theorems \ref{dec}  and \ref{IChandg}  that the $g$-polynomial $g(\Pi, t^2)$ is equal to the Poincar\'e polynomial of the fiber $\varphi^{-1}(0)$. 

By Theorem \ref{Poincar\'eFib} and Lemma \ref{dlgraphs}, the Poincar\'e polynomial of the fiber $\varphi^{-1}(0)$ is equal to 
\begin{equation*} 
\sum_{l} \mathrm{dim}H^{2l}(\varphi^{-1}(0))t^{2l} = \sum_{l} d_l(\Pi)(t^2-1)^{l},
\end{equation*}
where $d_l(\Pi)$ is the number of graphs in $\mathcal{G}^{\textcolor{red}{1}}_\Pi$ (cf. Lemma \ref{fibergraphs}) with $e(G_\Pi)+s(G_\Pi)-1+l$ edges; after a change of variables, we arrive at the following statement:
\begin{equation}\label{PoinPi}
 g(\Pi, t^2+1) = \sum_{l} d_l(\Pi)t^{2l}.
 \end{equation}

By Lemma \ref{facegraphs}, $\mathcal{G}^{\textcolor{red}{1}}_\Pi$ is the set of acyclic subgraphs $G\subset\mathbb{K}_k$, such that $G$ has $k$ vertices and contains a spanning tree rooted at the first vertex; we also have $e(G_\Pi)=0$ and $s(G_\Pi)=k$.   

Clearly, a graph $G\in \mathcal{G}^{\textcolor{red}{1}}_\Pi$ has no edges of the form $\overleftarrow{i1}$ for $2\leq i \leq k$, since otherwise this edge together with a paths from $i$ to $1$ would create a directed cycle in $G$; hence any graph $G\in \mathcal{G}^{\textcolor{red}{1}}_\Pi$ is contained in $\mathbb{G}_k\subset\mathbb{K}_k$.

We conclude that $d_l(\Pi)$ is equal to the number of acyclic  subgraphs $G$ of $\mathbb{G}_k$ with $k-1+l$ edges and $k$ vertices, such that $G$ is rooted at the first vertex. Now Theorem \ref{volosataya} follows from equation \eqref{PoinPi}.

Applying the same argument and using Lemma \ref{fibergraphs},  we obtain the following statement. 
\begin{corollary}\label{g-polyF} Let $F$ be a face of the Gale dual type-A root polytope $\Pi(r_{12},...,r_{k-1 k})$. Then 
$g(F, t+1) = \sum_{i\geq 0} \hat{g}_it^i$,
where $\hat{g}_i$ counts the number of graphs $G\in \mathcal{G}^{\textcolor{red}{1}}_F$ with $e(G_F)+s(G_F)-1+i$ edges and $k$ vertices.
\end{corollary}

Now we present a recursive formula for the $g$-polynomial of the Gale dual 
type-A root polytope.

\begin{proposition}\label{recgpoly}  Let $\Pi(r_{12},...,r_{k-1 k})$ be a 
Gale dual type-A root polytope and let $p(n,t)=1+t+t^2+...+t^{n-1}$. For any 
nonempty subset $J\subset\{2,...,k\}$ we denote by 
$\overline{J}=\{1,...,k\}\setminus J$ its complement  and  by $\Pi(r_{ij\in 
\overline{J}})$ the Gale dual root polytope corresponding to the root system 
$A_{|\overline{J}|-1}$  defined by the sequence of vectors 
$[\alpha_{ij}]_{i,j\in\overline{J}} \subset \A(r_{12},r_{13},...,r_{k-1 k})$ with 
$\alpha_{ij}$ repeated $r_{ij}$ times. Then
\begin{equation*}
g(\Pi(r_{12},...,r_{k-1 k}), t) = \sum_{\substack{J\subset\{2,...,k\} \\ J\neq\emptyset}}  (-1)^{|J|-1}g(\Pi(r_{ij\in \overline{J}}),t)\cdot\prod_{j\in J}p\big{(}\sum_{i\in\overline{J}}r_{ij},t\big{)}.
\end{equation*}
\end{proposition}
\begin{proof}
First, we perform a change of variables $t\to t+1$ in the recursion; by Theorem \ref{volosataya}, the polynomial $g(\Pi(r_{12},...,r_{k-1 k}), t+1)$ counts the number of edges in acyclic subgraphs of $\mathbb{G}_k$  with $k$ vertices, rooted at the first vertex. Note that any such graph has a vertex $j\in\{2,3,...,k\}$ that has no in-edges. 

Now fix a nonempty subset $J\subset\{1,2,...,k\}$. We claim that  
\begin{equation}\label{eq:recg}
g(\Pi(r_{ij\in \overline{J}}),t+1)\cdot\prod_{j\in J}p\big{(}\sum_{i\in\overline{J}}r_{ij},t+1\big{)} = \sum_{G\in\mathbb{G}^{'}_k} t^{e(G)-k+1},
\end{equation}
 where the sum is taken over the set $ \mathbb{G}^{'}_k $ of the connected acyclic subgraphs 
 $G\subset\mathbb{G}_k$  with $k$ vertices, rooted at the first vertex, such 
 that the vertices $j\in J$ of $G$ has no in-edges. 
 
 Indeed, let $G$ be a graph satisfying these conditions.
\begin{itemize}
\item A simple calculation shows that the polynomial $p\big{(}\sum_{i\in\overline{J}}r_{ij},t+1\big{)}$ counts  all out-edges of $G$ from the vertex $j\in J$.
\item We denote by $G\setminus J$ the graph obtained from $G$ by deleting all vertices labelled by $j\in J$ and all edges attached to these vertices. Then clearly, $G\setminus J$ is an acyclic subgraph of $\mathbb{G}_k$ with vertex set $\overline{J}$, rooted at the first vertex. 
\item Thus it follows from Theorem \ref{volosataya} that the edges of 
$G\setminus J$ are counted by the polynomial  $g(\Pi(r_{ij\in \overline{J}}),t+1)$.
\item Noting that any edge of $G$ is either an out-edge from some vertex $j\in J$, or is an edge of $G\setminus J$, we arrive at equation \eqref{eq:recg}. 
\end{itemize}
Now the recursion follows from the inclusion-exclusion principle.
\end{proof}
\begin{example}
For $k=3$, the recursive formula has the following simple form
\begin{multline*}
g(\Pi(r_{12},r_{13},r_{23}), t) = p({r_{12}+r_{23},t})\cdot g(\Pi(r_{13}), t) + \\  p{(r_{13}+r_{23}, t)}\cdot g(\Pi(r_{12}), t) -  
p{(r_{12},t)}\cdot p{(r_{13},t)}.
\end{multline*}
\end{example}

\section{The $h$-polynomial of the Gale dual type-A root polytope}\label{sec:hpol}
As explained in Example \ref{k=2space}, for $k=2$, we have 
$X(\Sigma_{\Pi(r_{12})}) = \mathbb{P}^{r_{12}-1}\times\mathbb{P}^{r_{12}-1}$, 
hence by Proposition \ref{Poincar\'eh-poly} the $h$-polynomial of the polytope 
$\Pi(r_{12})$ is equal to the Poincar\'e polynomial of the product of two 
projective spaces:
$$h(\Pi(r_{12}), t) = (1+t+t^2+...+t^{r_{12}-1})^2.$$
In this section, we calculate $h$-polynomials of Gale dual polytopes which are associated to the $A_{k-1}$ root system for $k\geq3$.

\subsection{$k=3$ case}
We start with the calculation of the $h$-polynomial of the polytope 
$\Pi(1,1,1)$, shown in Figure \ref{fig1}. This three-dimensional prism 
has\begin{itemize}
	\item  6 faces of dimension $0$, 
	\item 9 faces of dimension $1$, and 
	\item 2 faces of dimension $2$, which are all simplicial;
	\item  there are also 3 two-dimensional nonsimplicial faces, which are 
	squares. 
\end{itemize}
Using  \eqref{g-poly}, we obtain that the $g$-polynomial of a square equals  
$1+t$. 
Then it follows from \eqref{h-polyARB} that
$$h(\Pi(1,1,1), t) = (t-1)^3+6(t-1)^2+9(t-1)+2+3(t+1) = (t+1)^3.$$
Note that $(t+1)^3$ is the Poincar\'e polynomial of 
$\mathbb{P}^1\times\mathbb{P}^1\times\mathbb{P}^1$, and by Proposition 
\ref{Poincar\'eh-poly} it is equal to the $h$-polynomial of the octahedron, 
which is combinatorially dual to the three-dimensional cube. 

This corresponds to the fact that the prism $\Pi(1,1,1)$ has a simplicial 
subdivision $\widetilde{\Pi}(1,1,1)$ which does not add any vertices and 
divides every two-dimensional nonsimplicial face into two simplices by adding 
its diagonal (see Figure \ref{fig:nonprojres}). The resulting polytope 
$\widetilde{\Pi}(1,1,1)$ is combinatorially equivalent to the octahedron and 
the map of the corresponding projective toric varieties (cf. \S\ref{S2.3})
$$f: X(\Sigma_{\widetilde{\Pi}(1,1,1)})\to X(\Sigma_{{\Pi}(1,1,1)})$$
is small. More precisely, it is an isomorphism outside the three singular points of $X(\Sigma_{{\Pi}(1,1,1)})$ which correspond to the cones in $\Sigma_{\Pi(1,1,1)}$ over the nonsimplicial faces, and the fibers over these points are isomorphic to $\mathbb{P}^1$.
\begin{remark}
Note that this agrees with Theorem \ref{Poincar\'eFib} applied to the map $$\varphi: X(\A(1,1,1),(2,-1,-1)) \to X(\A(1,1,1), 0).$$ As it is shown in the Figure \ref{fig:fiber}, the corresponding morphism of fans $\widehat{\varphi}_\Sigma$ maps 2 three-dimensional and 1 two-dimensional cones in $\Sigma_{\widetilde{\Pi}(1,1,1)}$  to the interior of the cone over the square face of $\Pi(1,1,1)$. Thus by Theorem \ref{Poincar\'eFib}  the Poincar\'e polynomial of the fiber over the singular point should have the form $(t^2-1)+2$, which is indeed equal to the Poincar\'e polynomial of $\mathbb{P}^1$.
\end{remark}

\begin{figure}[H]
\begin{minipage}{.49\textwidth}
\centering
\begin{tikzpicture}[scale=1]
\path [draw=black, thick] (0,2) edge (3,2);
\path [draw=black, thick] (3,2) edge (1,1.2);
\path [draw=black, thick] (0,2) edge (1,1.2);
\path [draw=black, thick, dashed] (0,0.2) edge (3,0.2);
\path [draw=black, thick] (0,0.2) edge (1,-0.6);
\path [draw=black, thick] (1,-0.6) edge (3,0.2);
\path [draw=black, thick] (0,2) edge (0,0.2);
\path [draw=black, thick] (3,2) edge (3,0.2);
\path [draw=black, thick] (1,1.2) edge (1, -0.6);
\path [draw=orange, thick, dashed] (0,2) edge (3,0.2);
\path [draw=orange, thick] (0,0.2) edge (1,1.2);
\path [draw=orange, thick] (1,-0.6) edge (3,2);
\draw [fill] (1,-0.6) circle [radius=0.02];
\node [below] at (1,-0.55) {\tiny $\beta_{12}$};
\draw [fill] (0,0.2) circle [radius=0.02];
\node [below] at (0,0.2) {\tiny $\beta_{31}$};
\draw [fill] (3,0.2) circle [radius=0.02];
\node [below] at (3,0.2) {\tiny $\beta_{23}$};
\draw [fill] (1,1.2) circle [radius=0.02];
\node [right] at (0.9,1.1) {\tiny $\beta_{21}$};
\draw [fill] (0,2) circle [radius=0.02];
\node [left] at (0,2) {\tiny $\beta_{13}$};
\draw [fill] (3,2) circle [radius=0.02];
\node [right] at (3,2) {\tiny $\beta_{32}$};
\end{tikzpicture}
\vskip 14pt
\setlength{\belowcaptionskip}{-8pt}\caption{Refinement, combinatorially equivalent to the octahedron.}\label{fig:nonprojres}
\end{minipage}
\begin{minipage}{.49\textwidth}
\centering
\begin{tikzpicture}[scale=1]
\fill [fill=orange!10, ] (0,0.2) -- (2.1,0.8) -- (1,1.2);
\path [draw=black, thick] (0,2) edge (1,1.2);
\path [draw=black, thick] (0,0.2) edge (1,-0.6);
\path [draw=black, thick] (0,2) edge (0,0.2);
\path [draw=black, thick] (1,1.2) edge (1, -0.6);
\path [draw=orange, thick] (0,0.2) edge (1,1.2);
\draw [fill] (1,-0.6) circle [radius=0.02];
\node [below] at (1,-0.55) {\tiny $\beta_{12}$};
\draw [fill] (0,0.2) circle [radius=0.02];
\node [below] at (0,0.2) {\tiny $\beta_{31}$};
\draw [fill] (1,1.2) circle [radius=0.02];
\node [above] at (1.1,1.2) {\tiny $\beta_{21}$};
\draw [fill] (0,2) circle [radius=0.02];
\node [left] at (0,2) {\tiny $\beta_{13}$};
\path [draw=gray] (2.1,0.8) edge (0,2);
\path [draw=gray, dashed] (2.1,0.8) edge (0,0.2);
\path [draw=gray] (2.1,0.8) edge (1,-0.6);
\path [draw=gray] (2.1,0.8) edge (1,1.2);
\draw [fill] (2.1,0.8) circle [radius=0.02];
\node [above] at (2.1,0.8) {\tiny $0$};
\end{tikzpicture}
\vskip 14pt
\setlength{\belowcaptionskip}{-8pt}\caption{Calculation of the fiber over the singular point.}\label{fig:fiber}
\end{minipage}
\end{figure}
We note that the compact toric variety corresponding to this refinement is not 
projective.

It turns out that the $h$-polynomial of any Gale dual polytope obtained from 
the $A_2$  root system has the following elegant form, which is a special case 
of Theorem \ref{main:h-poly}.

\begin{lemma}\label{hpoly3} Let $\Pi(r_{12},r_{13},r_{23})$ be a Gale dual type-A root polytope and set $r_i=\sum_{j\neq i} r_{ij}$. Then the $h$-polynomial of $\Pi(r_{12},r_{13},r_{23})$  is equal to the Poincar\'e polynomial of the product of projective spaces $\mathbb{P}^{r_1-1}\times \mathbb{P}^{r_2-1}\times \mathbb{P}^{r_3-1}:$
$$h(\Pi(r_{12},r_{13},r_{23}), t) = \prod_{i=1}^3(1+t+t^2+...+t^{r_i-1}).$$ 
\end{lemma}

\subsection{Arbitrary $k$}
\begin{theorem}\label{main:h-poly} 
Let $\Pi(r_{12},..,r_{k-1 k})$ be a Gale dual type-A root polytope and let $r_i=\sum_{j\neq i}r_{ij}$. Given a subset $\lambda\subset\{1,...,k\}$, we denote by $\Pi({r}_{ij\in\lambda})$ the Gale dual  root polytope corresponding to the root system $A_{|\lambda|-1}$ defined by  the sequence of vectors $[\alpha_{ij}]_{i,j\in\lambda} \subset \A(r_{12},...,r_{k-1 k})$ with $\alpha_{ij}$ repeated $r_{ij}$ times. Then 
\begin{multline*}
h(\Pi(r_{12},...,r_{k-1 k}),t) = \\ \prod_{i=1}^k (1+t+...+t^{r_i-1}) - 
 \sum_{\substack{(\lambda_1,...,\lambda_p)\vdash\{1,...,k\} \\ p\geq 2, \, |\lambda_m|\geq 2}}t^{\sum\limits_{m<n}\sum\limits_{i\in \lambda_m}\sum\limits_{j\in \lambda_n}r_{ij}}\prod_{m=1}^ph(\Pi({r}_{ij\in\lambda_m}),t),
\end{multline*}
where the sum is taken over the partitions $\underline{\lambda}=(\lambda_1,...,\lambda_p)$ of the set $\underline{k}=\{1,...,k\}$ that have at least $2$ parts and do not have parts of cardinality $1$.
\end{theorem}
\begin{proof}
First, we perform a change of variables $t\to t+1$ and rewrite the equation from Theorem \ref{main:h-poly} in the following form
\begin{equation}\label{eq:h-poly}
\prod_{i=1}^k ((t+1)^{r_i}-1)/t =  \sum_{\substack{\underline{\lambda}\vdash\underline{k} \\  |\lambda_m|\geq 2}}(t+1)^{\sum\limits_{m<n}\sum\limits_{i\in \lambda_m}\sum\limits_{j\in \lambda_n}r_{ij}}\prod_{m=1}^ph(\Pi({r}_{ij\in\lambda_m}),t+1).
\end{equation}

We can interpret the left-hand side of \eqref{eq:h-poly} in terms of graphs as follows.
\begin{lemma}\label{LHS12}
 Let $r_i=\sum_{j\neq i}r_{ij}$, then
$$\prod_{i=1}^k ((t+1)^{r_i}-1)/t = \sum_{i\geq 0} \hat{p}_it^i,$$
where $\hat{p}_i$ counts the number of subgraphs of $\mathbb{K}_k$ with $k+i$ edges and $k$ vertices, which have an out-edge from each vertex.
\end{lemma}
\begin{proof}
Note that the $i^{\mathrm{th}}$ factor in the product counts the number of out-edges from the $i^{\mathrm{th}}$ vertex in a subgraph of $\mathbb{K}_k$. 
\end{proof}

Now we describe the right-hand side of \eqref{eq:h-poly} in terms of graphs. 
\begin{theorem}\label{hgraphs}
We have the following combinatorial description of the $ h $-polynomial of the Gale dual type-A root polytope $\Pi(r_{12},...,r_{k-1 k})$:
\begin{equation}
h(\Pi(r_{12},...,r_{k-1 k}),t+1) =\sum_{G\in \mathbb{K}^{'}_k}t^{e(G)-k},
\end{equation}
where the sum is taken over the set $ \mathbb{K}^{'}_k $ of subgraphs of the complete graph $\mathbb{K}_k$ which have $k$ vertices, at least one cycle and which are rooted at the first vertex. 
\end{theorem}
\begin{proof}
Using equations \eqref{h-polyARB}, \eqref{codimcone} and Corollary \ref{g-polyF}, we obtain that
\begin{multline}\label{h-polyint}
h(\Pi(r_{12},...,r_{k-1 k}),t+1)  \overset{{\eqref{h-polyARB}}}= \sum_{F\lneq\Pi}g(F, t+1) t^{d-\mathrm{dim}(F)-1}  \overset{{\eqref{codimcone}}}= \\ \sum_{F\lneq\Pi}g(F, t+1) t^{e(G_F)+s(G_F)-k-1} \overset{\ref{g-polyF}}= 
 \sum_{F\lneq\Pi} \sum_{G\in \mathcal{G}^{\textcolor{red}{1}}_F} t^{e(G)-k} = \sum_{G\subset \mathbb{K}^{'}_k}t^{e(G)-k}
\end{multline}
\end{proof}

We can thus rewrite the right-hand side of equation \eqref{eq:h-poly} as 
\begin{equation}\label{rhsof12}
\sum_{\substack{\underline{\lambda}\vdash\underline{k} \\  |\lambda_i|\geq 2}} \sum_{G\in\mathbb{K}^{''}_k}t^{e(G)-k},
\end{equation}
where the last sum is taken over the set $ \mathbb{K}^{''}_k $ of subgraphs $G\subset \mathbb{K}_k$ which contain subgraphs $G_{\lambda_1},...,G_{\lambda_p}$ satisfying the following conditions: 
\begin{itemize}
\item the set of vertices of $G_{\lambda_i}$ is $\lambda_i\subset\{1,...,k\}$;
\item $G_{\lambda_i}$ contains a directed cycle;
\item $G_{\lambda_i}$ is rooted at the smallest vertex $v_i\in\lambda_i$;
\item all edges in $G$ between the sets of vertices $\lambda_i$ and $\lambda_j$ are in one direction: from $\lambda_i$ to $\lambda_j$, if the minimum of $\lambda_i\cup \lambda_j$ is equal to the minimum of $\lambda_i$.
\end{itemize} 
The last condition follows from the observation that  the exponent of $(t+1)$ on the right-hand side of equation \eqref{eq:h-poly} is half of the number of edges in the graph $\mathbb{K}_k$ between the sets of vertices $\lambda_1,...,\lambda_p$. 

In fact, to prove Theorem \ref{main:h-poly}, we need to choose the root vertex of $G_{\lambda_i}$ and the direction of the edges between the sets of vertices $\lambda_i$ and $\lambda_j$ in a more subtle way (for the result, see Lemma \ref{conditions}). 
The following construction allows us to perform a change of the root vertex.

\noindent\textbf{Construction [roots]}:\label{construction:roots} Let $ G $ be an acyclic subgraph of $\mathbb{K}_k$, which have $k$ vertices and are rooted at the first vertex. Change the direction of all edges in all directed paths from the $i^{\mathrm{th}}$ vertex to the first vertex, we obtain a bijection between two sets. 
\begin{lemma}\label{redvertex}
This construction induces a one-to-one correspondence between the acyclic subgraphs of $\mathbb{K}_k$, which have $k$ vertices and are rooted at the first vertex and the acyclic subgraphs of $\mathbb{K}_k$  with $k$ vertices, which are rooted at the $i^{\mathrm{th}}$ vertex. 
\end{lemma}

Throughout this section, we will use the following construction. 

\noindent\textbf{Construction [cycles]:}\label{construction:cycle} Let $G$ be a subgraph of $\mathbb{K}_k$ which contain at least one directed cycle. Removing all edges from $G$, that are not contained in any directed cycle, we obtain the maximal naked subgraph $G_F\subset G$. We say that a connected component of $G_F$ is \textit{non-trivial}, if it contains at least one edge; let $s$ be the number of non-trivial connected components of $G_F$. We order the non-trivial components $C_1\prec...\prec C_s$ in such way that for any $1\leq i<j \leq s$ the minimal vertex of $C_i \cup C_j$ is contained in $C_i$. 

\noindent\textbf{Notation:} We will refer to these ordered non-trivial connected components of $G_F$ as  \textit{cycles of} $G$.

\begin{lemma}\label{conditions}
The right-hand side of equation \eqref{eq:h-poly} can be expressed as the sum \eqref{rhsof12}, where now the last sum runs over graphs $G\subset \mathbb{K}_k$, which contains subgraphs $G_{\lambda_1},...,G_{\lambda_p}$ satisfying the following four conditions: 
\begin{enumerate}\label{conditions} 
\item the set of vertices of $G_{\lambda_i}$ is $\lambda_i\subset\{1,...,k\}$;
\item $G_{\lambda_i}$ contains a directed cycle;
\item $G_{\lambda_i}$ is rooted at the smallest vertex  in the smallest cycle of $G$ contained in $ G_{\lambda_i}$;
\item all edges in $G$ between the sets of vertices $\lambda_i$ and $\lambda_j$ are given in one direction: from $\lambda_i$ to $\lambda_j$, if the smallest cycle in $G_{\lambda_i}\cup G_{\lambda_j}$ is contained in $G_{\lambda_i}$.
\end{enumerate} 
\end{lemma}
\begin{proof}
First, note that any cycle $C_i$ of $G$ contains a directed path between any two vertices, hence all vertices of $C_i$ are contained in exactly one subgraph $G_{\lambda_j}\subset G$. 
 
Let $G_{\lambda_i}$ be a subgraph of $G$, such that its root vertex, the 
minimum $v_i$ of $\lambda_i$, is not contained in any cycle of $G$; denote by 
$C^{\lambda_{i}}$ the smallest cycle of $G$, which is contained in 
$G_{\lambda_i}$. We proceed as described in Lemma \ref{redvertex}, with the 
vertices of acyclic graphs replaced by the cycles of $G$. 
More precisely, we change the direction of all paths from the vertices of 
$C^{\lambda_{i}}$ to the root vertex $v_i$; we thus obtain a graph rooted at 
the smallest vertex of the smallest cycle contained in $ G_{\lambda_i}$, i.e. 
satisfying condition (3).

Similarly, to obtain condition (4) we consider the pairs of subgraphs $G_{\lambda_i}, G_{\lambda_j}$ of $G$, such that the minimum of $\lambda_i\cup \lambda_j$ is not contained in any cycle of $G$. We pick the smallest cycle of $G$, which is contained in  $G_{\lambda_i}\cup G_{\lambda_j}$ and change the 
 direction of all paths from the vertices of $G_{\lambda_i}$ to $G_{\lambda_j}$, if needed. 
\end{proof}

\noindent\textbf{Notation:} We introduce the notation L\eqref{eq:h-poly} for the set of subgraphs $G\subset\mathbb{K}_k$ with $k$ vertices that contain an out-edge from each vertex, and R\eqref{eq:h-poly} for the 
set of subgraphs $G\subset\mathbb{K}_k$ satisfying conditions (1)-(4) from Lemma \ref{conditions}.

We showed in Lemmas \ref{LHS12} and \ref{conditions} that the left-hand and the right-hand sides of equation \eqref{eq:h-poly} count graphs (with fixed number of edges) from the sets L\eqref{eq:h-poly} and R\eqref{eq:h-poly}, correspondingly.
 Hence to proof equation \eqref{eq:h-poly}, we need to construct a bijection between two sets of graphs L\eqref{eq:h-poly} and R\eqref{eq:h-poly}.

We start by observing that any graph satisfying conditions (1)-(4) from Lemma \ref{conditions} is clearly a subgraph of $\mathbb{K}_k$ with $k$ vertices, and has an out-edge from each vertex, and thus R\eqref{eq:h-poly} is a subset of L\eqref{eq:h-poly}.

Now let $G$ be an element of L\eqref{eq:h-poly}. To obtain a map from L\eqref{eq:h-poly} to R\eqref{eq:h-poly}, we 
need to associate to $G$ a partition $(\lambda_1,...,\lambda_p)$  of the set $\{1,...,k\}$ and find subgraphs $G_{\lambda_1},...,G_{\lambda_p} \subset G$, satisfying conditions (1)-(4) from Lemma \ref{conditions}.

Let $C_1\prec...\prec C_s$ be the cycles of $G$.
We set $\lambda_0=\emptyset$, and for $j\geq 0$ we iterate the following procedure to obtain a partition $(\lambda_1,...,\lambda_p)$ of $\{1,...,k\}$.
 \begin{itemize}
 \item Let $\lambda_j$ be a subset of $\{1,...,k\}$, such that each $v_i\subset\{1,...,k\}$ is either contained in $\lambda_j$ or does not intersect $\lambda_j$. We introduce the notation $G{\setminus \lambda_j}$ for the graph obtained from $G$  by deleting all vertices labelled by $\lambda_0\cup\lambda_1\cup...\cup\lambda_j$ and all edges attached to these vertices. We denote by $C^{\lambda_{j+1}}$ the smallest cycle of $G$, which is contained in $G{\setminus \lambda_j}$. 
 \item Let $\lambda_{j+1}\subset\{1,...,k\}$  be the set of vertices $i$ of the graph $G{\setminus\lambda_j}$, such that $C^{\lambda_{j+1}}$ is reachable from $i$, i.e. there is a directed path in $G{\setminus\lambda_j}$ from the vertex $i$ to a vertex in a subgraph $C^{\lambda_{j+1}}$.
\item As observed above, for any $1\leq i\leq s$, the connected graph $C_i$ contains a directed path between any two vertices, hence the set of vertices of $C_i$ is either contained in $\lambda_{j+1}$ or does not intersect $\lambda_{j+1}$. We can thus repeat the procedure, until we arrive at the empty graph $G\setminus\lambda_p$ for some $p$. 
 \end{itemize}
 
We make the following observation.
\begin{lemma}
Let $\lambda_1,...,\lambda_p$ be subsets of $\{1,...,k\}$ associated to a graph $G$ as above.  Then $(\lambda_1,...,\lambda_p)$ is a partition of the set $\{1,...,k\}$.
\end{lemma}
\begin{proof} 
It follows from the construction that all vertices of non-trivial cycles of $G$ are contained in the union $\lambda_1\cup ...\cup \lambda_p$.
We have to show that if the maximal naked graph  $G_F\subset G$ has a trivial connected component which consists of a single vertex $v$, then $v$ is an element of $\lambda_i$ for some $1\leq i \leq p$. 

Recall that $G$ contains an out-edge from the vertex $v$, and thus there is a directed path from $v$ to some non-trivial cycle $C_j$ of $G$; then it follows from the construction that $v$ belongs to the same subset $\lambda_i\subset\{1,...,k\}$ as all vertices of $C_j$. 
\end{proof}
We denote by $G_{\lambda_i}$ the subgraph of $G$ obtained by taking vertices from $\lambda_i\subset\{1,...,k\}$ and all edges between them. Then clearly, the collection of subgraphs $G_{\lambda_1},...,G_{\lambda_p}$ of $G$ satisfy conditions (1)-(4) from Lemma \ref{conditions}, and thus we obtain a map from the set L\eqref{eq:h-poly} to the set R\eqref{eq:h-poly}. 
 It is easy to check that this map is a bijection, and thus Theorem \ref{main:h-poly} follows.
\end{proof}

\section{Multiplication}\label{sec:multi}

In this section, we introduce a ring structure on the intersection cohomology 
of the affine toric variety $X(\A(r_{12},...,r_{k-1 k}),0)$ induced by the 
small resolution of singularities.

\subsection{Other chambers}\label{S6.1}

It turns out that we could have carried out the arguments of 
\S\ref{S4.2}-\S\ref{S4.4}  for any chamber in the chamber complex 
$Ch(\A(r_{12},...,r_{k-1 
k}))$.

\begin{proposition}\label{allsmall}
	Let $\theta$ be a generic point in $\mathbb{N}\A(r_{12},...,r_{k-1 k})$. 
	Then the morphism of toric varieties $$\varphi_\theta: 
	X(\A(r_{12},...,r_{k-1 
	k}), \theta)\to X(\A(r_{12},...,r_{k-1 k}),0)$$ is small.
\end{proposition}
\begin{proof}
	The proof follows the logic of the proof of Theorem \ref{smallmor} . We 
	repeat the argument with only minor changes, thus omitting the details.
	
	As explained in the proof of Lemma \ref{lemma:trees}, each cone 
	$\cone(\alpha^I)$ for a basis index set 
	$I\in\mathrm{BInd}(\A(r_{12},...,r_{k-1 k}), \mathfrak{c})$ correspond to a 
	spanning tree in $\mathbb{K}_k$; we introduce the notation 
	$\mathrm{Trees}(\theta)$ for the set of these trees. 
	
	Let $F$ be a face of $\Pi$, and let $G_F$ be the corresponding naked graph 
	(cf. Lemma \ref{facegraphs}).
	Denote by $\mathcal{G}^{\theta}_F$  the set of connected subgraphs of 
	$\mathbb{K}_k$  with $k$ vertices that have the same directed cycles as the 
	naked graph $G_F$ and that contain at leas one tree from 
	$\mathrm{Trees}(\theta)$.
	Then (cf. Theorem \ref{Poincar\'eFib} and Lemma \ref{dlgraphs}) for any 
	point  $y$ in $\mathcal{O}(\cone(F))\subset X$ the Poincar\'e polynomial of 
	the fiber $\varphi_\theta^{-1}(y)$ is equal to
	$$\sum_{l\geq 0} d_l(F)(t^2-1)^l,$$
	where $d_l(F)$ is the number of graphs in $\mathcal{G}^{\theta}_F$ with 
	$e(G_F)+s(G_F)-1+l$ edges. 
	
	Repeating the dimension estimates from the proof of Theorem \ref{smallmor}, 
	we arrive at our statement.
\end{proof}
\subsection{A ring structure on the intersection cohomology}
Applying Theorem \ref{dec} to the small morphism $\varphi_\theta$ from 
Proposition \ref{allsmall}, we obtain the following statement.
\begin{corollary}\label{psi}
	Let $\theta$ and $\varphi_\theta$ be as in Proposition \ref{allsmall}. 
	There is an isomorphism 
	$$\psi_\theta: H^*(\varphi_\theta^{-1}(0)) \xrightarrow{\sim} IH^*(X(\A(r_{12},...,r_{k-1 
	k}),0))$$
	between the cohomology groups of the fiber $\varphi_\theta^{-1}(0)$ and the 
	intersection cohomology groups of the affine cone $X(\A(r_{12},...,r_{k-1 
	k}),0)$.
	The isomorphism $\psi_\theta$ induces a ring structure on the intersection 
	cohomology of $X(\A(r_{12},...,r_{k-1 k}),0)$.
	\end{corollary}

The cohomology rings $H^*(\varphi_\theta^{-1}(0))$ were studied in \cite{HSt}. 
In particular, it was shown that they do not depend on the choice of the 
elements $\theta\in\ak^*$ (cf. Theorem \ref{algrelations}), and thus the 
induced ring structure on the intersection cohomology of  
$X(\A(r_{12},...,r_{k-1 k}),0)$ is canonical.

\begin{theorem}\cite[Theorem 8.3]{HSt}\label{algrelations}
	Let $\theta$ be a generic point in $\mathbb{N}\A(r_{12},...,r_{k-1 k})$. 
	The cohomology ring of the fiber $\varphi_\theta^{-1}(0)$ is canonically 
	isomorphic to the quotient of 
	$\mathbb{C}[\varepsilon_i-\varepsilon_j, 1\leq i<j\leq k]$ by the ideal 
	generated by the polynomials 
\begin{equation}\label{HStrelations}
		p_{D}(\varepsilon_1,\dots,\varepsilon_k)=\prod_{\substack{i\in 
		D_1, j\in D_2}}(\varepsilon_i-\varepsilon_j)^{r_{ij}},
\end{equation} where $D=D_1\sqcup 
			D_2$ runs 
	over all nontrivial partitions of the set $\{1,2,...,k\}$. 
\end{theorem}

In the remainder of this section, we explain how these relations appear in our 
graphical formalism.

We begin by describing the reducible variety $\varphi_\theta^{-1}(0)$.  
We introduce the notation $\mathrm{Top}(\theta)$ for the subset of graphs in 
$\mathcal{G}^{\theta}_\Pi$ with maximal number of edges, which is equal to  
$\sum_{i<j}r_{ij}$. Using Lemma 
\ref{facegraphs}, we can reformulate Lemma 2.1.11 from \cite{Yau} in terms of 
graphs in the following form.

\begin{proposition}\label{toricfib}
	(i) The fiber of $\varphi_\theta$ over a point $y_\Pi\in 
	X(\A(r_{12},...,r_{k-1 k}),0)$ is a connected reducible variety, 
	whose irreducible components are toric varieties parametrized by the 
	elements in $\mathrm{Top}(\theta)$. 
	
	(ii) Let $G\in\mathrm{Top}(\theta)$, and denote the sequence of its edges 
	by $\mathcal{A}_G\subset \A(r_{12},...,r_{k-1 k})$; then the irreducible 
	component of $\varphi_\theta^{-1}(0)$ associated to $G$ is the toric 
	variety  $X(\mathcal{A}_G,\theta)$, which has  dimension 
	$\sum_{i<j}r_{ij}-k+1$. 
	
	(iii) For $G_1,...,G_m\in\mathrm{Top}(\theta)$,  denote by $G_1\cap...\cap 
	G_m$ the graph obtained by taking all common edges of $G_1,...,G_m$. Then 
	the  
	irreducible components  
	$X(\mathcal{A}_{G_1},\theta),...,X(\mathcal{A}_{G_m},\theta)$ are glued 
	along the toric variety $X(\mathcal{A}_{G_1\cap...\cap G_m},\theta)$.
\end{proposition}

\begin{remark}
	Unlike their cohomology rings, the varieties $\varphi_\theta^{-1}(0)$ are 
	not necesserely isomorphic for different $\theta$s. For example, let 
	$\theta_1 = (3,-1,-1,-1) \text{\,\,\, and \,\,\, } \theta_\mathrm{rand} = 
	(2,-3,2,-1)$ be two points in  the chamber complex $Ch(\A(1,1,1,1,1,1))$. A 
	simple calculation shows that $\varphi_{\theta_1}^{-1}(0)$ has 6 
	irreducible components that have Poincar\'e polynomials $1+2t^2+2t^4+t^6$, 
	while $\varphi_{\theta_\mathrm{rand}}^{-1}(0)$ has 6 irreducible 
	components, two of which are isomorphic to $\mathbb{P}^3$.
\end{remark}

To describe the ring structure on $H^*(\varphi_\theta^{-1}(0))$, we need one 
more ingredient: a general statement about relations in the cohomology 
ring of compact toric varieties. These are well-known; we will present them in 
the formalism of \cite{SzV}. We will use the notation of 
\S\ref{S2.1}-\S\ref{S2.2}.

Let $\theta$ be a weight in $\Gamma_{\mathfrak{a}}^*$. By the Chern-Weil 
construction, every polynomial on $\mathfrak{a}$ gives rise to a characteristic 
class of the toric variety $X(\A,\theta)$. Thus there is  a Chern-Weil 
homomorphism 
$\chi: \mathrm{Sym}(\mathfrak{a}^*)\to H^*(X(\A,\theta))$ from the polynomials 
on $\mathfrak{a}$ to the cohomology of $X(\A,\theta)$. In particular, for any 
vector $\alpha\in\Gamma_{\mathfrak{a}}^*$, $\chi(\alpha)$ is an element of 
$H^2(X(\A,\theta))$.

Now let $\A$ be a  sequence which lies in an open half space 
of the vector space $\mathfrak{a}^*$, and assume that $ \theta $ is generic. 
Then $ X(\A,\theta) $ is a projective orbifold, and for every 
polynomial $Q\in \mathrm{Sym}(\mathfrak{a}^*)$, one can write explicit formulas 
for  the intersection numbers $\int_{X(\A,\theta)}\chi(Q)$ using the 
Jeffrey-Kirwan residue \cite{BV, SzV}. In this paper, we will only need to 
integrate  some particularly simple classes, described below.

\begin{proposition}\cite[Proposition 2.3]{SzV}\label{toricint}
	Let $\A$ be a  sequence which lies in an open half space 
	of $\mathfrak{a}^*$, and let 
	$\theta$ be a generic point in $\Gamma_\mathfrak{a}^*$. 
	Then, for any set of indices $J\subset \{1,...,n\}^{n-d}$ we have  
	$$\int\limits_{X(\A,\theta)}\prod_{j\in J}\chi(\alpha_j) = 0,$$ if the 
	complement of $\{\alpha_j\}_{j\in J}$ in $\A$ does not span 
	$\mathfrak{a}^*$.
\end{proposition}

Now we apply this statement to our situation.
Let $\theta\in\Gamma_{\mathfrak{a}}^*$ be generic, 
$G\in\mathrm{Top}(\theta)$ and $\mathcal{A}_G\subset\A(r_{12},...,r_{k-1 k})$ 
be as in Proposition \ref{toricfib}. Given  a nontrivial partition $D=D_1\sqcup 
D_2$ of the set $\{1,2,...,k\}$, denote by $\mathcal{A}[{D}]$ the subsequence 
$[\alpha_{ij}]_{\substack{i\in D_1, j\in D_2}}\subset\A(r_{12},...,r_{k-1 k})$, 
where the element $\alpha_{ij}\in\mathcal{A}[{D}]$ is repeated $r_{ij}$ times.

Clearly,  the sequence $\mathcal{A}_G\cap ((\A(r_{12},...,r_{k-1 
k})\setminus\mathcal{A}[{D}])$
corresponds to a subgraph of $\mathbb{K}_k$, which doest not contain any tree 
on $k$ vertices. Then it follows from Proposition \ref{toricint} that the 
product $$\prod_{\substack{i\in D_1 \\ j\in D_2}}\chi(\alpha_{ij})^{r_{ij}}\in 
H^*(X(\mathcal{A}_G, \theta))$$ is zero for any $G\in\mathrm{Top}(\theta)$, hence it is equal to zero in $ H^*(\varphi^{-1}_\theta(0))$.  We have thus 
reproduced the Hausel-Sturmfels relations given in \eqref{HStrelations}.

\end{document}